\documentclass[reqno]{amsart}%%%%last updated on 12/16/09
\usepackage{amsmath, amssymb, amsthm, epsfig}
\usepackage{hyperref, latexsym}
\usepackage{url} 
\usepackage{setspace}

\DeclareMathOperator{\sgn}{\mathrm{sgn}}

\begin{document}
 \bibliographystyle{plain}

 \newtheorem{theorem}{Theorem}%[section]
 \newtheorem{lemma}[theorem]{Lemma}
 \newtheorem{proposition}[theorem]{Proposition}
 \newtheorem{corollary}[theorem]{Corollary}
 \theoremstyle{definition}
 \newtheorem{definition}[theorem]{Definition}
 \newtheorem{example}[theorem]{Example}
 \theoremstyle{remark}
 \newtheorem{remark}[theorem]{Remark}
 \newcommand{\mc}{\mathcal}
 \newcommand{\A}{\mc{A}}
 \newcommand{\B}{\mc{B}}
 \newcommand{\cc}{\mc{C}}
 \newcommand{\D}{\mc{D}}
 \newcommand{\E}{\mc{E}}
 \newcommand{\F}{\mc{F}}
 \newcommand{\G}{\mc{G}}
 \newcommand{\sH}{\mc{H}}
 \newcommand{\I}{\mc{I}}
 \newcommand{\J}{\mc{J}}
 \newcommand{\K}{\mc{K}}
 \newcommand{\lL}{\mc{L}}
 \newcommand{\M}{\mc{M}}
 \newcommand{\nn}{\mc{N}}
 \newcommand{\rr}{\mc{R}}
 \newcommand{\sS}{\mc{S}}
 \newcommand{\U}{\mc{U}}
 \newcommand{\X}{\mc{X}}
 \newcommand{\Y}{\mc{Y}}
 \newcommand{\C}{\mathbb{C}}
 \newcommand{\R}{\mathbb{R}}
 \newcommand{\N}{\mathbb{N}}
 \newcommand{\Q}{\mathbb{Q}}
 \newcommand{\Z}{\mathbb{Z}}
 \newcommand{\csch}{\mathrm{csch}}
 \newcommand{\tF}{\widehat{F}}
 \newcommand{\tG}{\widehat{G}}
 \newcommand{\tH}{\widehat{H}}
 \newcommand{\tf}{\widehat{f}}
 \newcommand{\ug}{\widehat{g}}
 \newcommand{\wg}{\widetilde{g}}
 \newcommand{\uh}{\widehat{h}}
 \newcommand{\wh}{\widetilde{h}}
 \newcommand{\wl}{\widetilde{l}}
 \newcommand{\tk}{\widehat{k}}
 \newcommand{\tK}{\widehat{K}}
 \newcommand{\tl}{\widehat{l}}
 \newcommand{\tL}{\widehat{L}}
 \newcommand{\tm}{\widehat{m}}
 \newcommand{\tM}{\widehat{M}}
 \newcommand{\tp}{\widehat{\varphi}}
 \newcommand{\tq}{\widehat{q}}
 \newcommand{\tT}{\widehat{T}}
 \newcommand{\tU}{\widehat{U}}
 \newcommand{\tu}{\widehat{u}}
 \newcommand{\tV}{\widehat{V}}
 \newcommand{\tv}{\widehat{v}}
 \newcommand{\tW}{\widehat{W}}
 \newcommand{\ba}{\boldsymbol{a}}
 \newcommand{\bal}{\boldsymbol{\alpha}}
 \newcommand{\bx}{\boldsymbol{x}}
 \newcommand{\p}{\varphi}
 \newcommand{\f}{\frac52}
 \newcommand{\g}{\frac32}
 \newcommand{\h}{\frac12}
 \newcommand{\hh}{\tfrac12}
 \newcommand{\ds}{\text{\rm d}s}
 \newcommand{\dt}{\text{\rm d}t}
 \newcommand{\du}{\text{\rm d}u}
 \newcommand{\dv}{\text{\rm d}v}
 \newcommand{\dw}{\text{\rm d}w}
  \newcommand{\dz}{\text{\rm d}z}
 \newcommand{\dx}{\text{\rm d}x}
 \newcommand{\dy}{\text{\rm d}y}
 \newcommand{\dl}{\text{\rm d}\lambda}
 \newcommand{\dmu}{\text{\rm d}\mu(\lambda)}
 \newcommand{\dnu}{\text{\rm d}\nu(\lambda)}
\newcommand{\dnus}{\text{\rm d}\nu_{\sigma}(\lambda)}
 \newcommand{\dlnu}{\text{\rm d}\nu_l(\lambda)}
 \newcommand{\dnnu}{\text{\rm d}\nu_n(\lambda)}
\newcommand{\sech}{\text{\rm sech}}
\newcommand{\CC}{\mathbb{C}}
\newcommand{\NN}{\mathbb{N}}
\newcommand{\RR}{\mathbb{R}}
\newcommand{\ZZ}{\mathbb{Z}}
\newcommand{\thp}{\theta^+}
\newcommand{\thpn}{\theta^+_N}
\newcommand{\vthp}{\vartheta^+}
\newcommand{\vthpn}{\vartheta^+_N}
\newcommand{\ft}[1]{\widehat{#1}}
\newcommand{\support}[1]{\mathrm{supp}(#1)}
\newcommand{\gplus}{G^+_\lambda}
\newcommand{\ph}{\mathcal{H}}
\newcommand{\godd}{G^o_\lambda}
\newcommand{\mplus}{M_\lambda^+}
\newcommand{\lplus}{L_\lambda^+}
\newcommand{\modd}{M_\lambda^o}
\newcommand{\lodd}{L_\lambda^o}
\newcommand{\sgp}{x_+^0}

 \def\today{\ifcase\month\or
  January\or February\or March\or April\or May\or June\or
  July\or August\or September\or October\or November\or December\fi
  \space\number\day, \number\year}

\title[Approximations to the truncated Gaussian]{Bandlimited approximations to the truncated Gaussian and applications}
\author[Carneiro and Littmann]{Emanuel Carneiro and  Friedrich Littmann}

\date{\today}
\subjclass[2000]{Primary 41A30, 41A52. Secondary 41A05, 41A44, 42A82}
\keywords{Truncated Gaussian, exponential type, extremal functions, onesided best approximation, tempered distributions.}

\address{IMPA - Instituto de Matem\'{a}tica Pura e Aplicada, Estrada Dona Castorina, 110, Rio de Janeiro, Brazil 22460-320.}
\email{carneiro@impa.br}
\address{Department of mathematics, North Dakota State University, Fargo, ND 58105-5075.}
\email{friedrich.littmann@ndsu.edu}

\begin{abstract} 
In this paper we extend the theory of optimal approximations of functions $f: \R \to \R$ in the $L^1(\R)$-metric by entire functions of prescribed exponential type (bandlimited functions). We solve this problem for the truncated and the odd Gaussians using explicit integral representations and properties of truncated theta functions obtained via the maximum principle for the heat operator. As applications, we recover most of the previously known examples in the literature and further extend the class of truncated and odd functions for which this extremal problem can be solved, by integration on the free parameter and the use of tempered distribution arguments. This is the counterpart of the work \cite{CLV}, where the case of even functions is treated.
\end{abstract}

\maketitle

\numberwithin{equation}{section}

\tableofcontents

\section{Introduction}%%%%%%%%%%%%%%%%%%%%%%%%%%%%%%%%%%%%%%%%%%%%%%%%%%%%%%%%%%%%%%%%%%%%%%%%%%%%%%%%%%%%%%%%%%%

An entire function $K:\C \to \C$ is of {\it exponential type} at most $2\pi \delta$ if, for every $\epsilon >0$, there exists a positive constant $C_{\epsilon}$, such that the inequality 
\begin{equation*}\label{intro0}
|K(z)| \leq C_{\epsilon} e^{(2\pi \delta + \epsilon) |z|}
\end{equation*}
holds for all $z \in \C$.  These functions are also referred to as bandlimited functions, since their distributional Fourier transforms are compactly supported in $[-\delta,\delta]$, as a consequence of the Paley-Wiener theorem.

\medskip

For a given function $f: \R \to \R$ we consider the problem of finding an entire function $K(z)$ of exponential type at most $2 \pi \delta$, such that the integral 
\begin{equation}\label{BS1}
 \int_{-\infty}^{\infty} |f(x) - K(x)|\, \dx
\end{equation}
is minimized.  Such entire function is called a best two-sided approximation. As a variant of this problem, we may impose the additional condition that $K(z)$ is real-valued on $\R$ and satisfies $f(x) \geq K(x)$ for all $x \in \R$.  In this case a function $K(z)$ that minimizes the integral (\ref{BS1}) is called an extreme minorant of $f(x)$ (or lower one-sided approximation).  Extreme majorants (upper one-sided approximations) are defined analogously.

\medskip

For the history of the one-sided approximations we refer to the survey \cite{V}. The two-sided problem was considered by Bernstein, Akhiezer, Krein, Nagy, and others, since at least 1938. In particular Krein \cite{K} in 1938 and Nagy \cite{Na} in 1939 published seminal papers solving this problem for a wide class of functions. A recent paper of Ganzburg \cite{Ga} investigates the two-sided problem following the classical method of Sz.-Nagy (cf. \cite[Chapter 7]{Sh}).

\medskip

Applications of the solutions to these problems include sieving inequalities \cite{M, S2, V}, Hilbert-type inequalities \cite{CV2, GV, Lit, MV, V}, Erd\"{o}s-Tur\'{a}n discrepancy inequalities \cite{CV2, LV, V}, optimal approximations of periodic functions by trigonometric polynomials \cite{Car, CV2, CV3, V}, Tauberian theorems \cite{GV} and, more recently, bounds for the Riemann zeta-function and its argument function on the critical line under the Riemann hypothesis \cite{CS, GG}. The extremal problem in higher dimensions, with applications, is considered in \cite{BMV, HV}. Approximations in $L^p$-norms with $p\neq 1$ are treated, for instance, in \cite{GL}. Other works on the subject include \cite{Ga2, Ga3, Lit2, Lit0, Lit3}.

\medskip

In \cite{CLV} the solution of the extremal problem (\ref{BS1}) is obtained for the Gaussian 
\begin{equation*}
G_{\lambda}(z) = e^{-\pi \lambda z^2}\,,
\end{equation*}
where $\lambda>0$ is a free parameter. This result together with tempered distribution arguments for even functions provides the solution of the extremal problem for a large class of even functions. Number-theoretical applications such as Hilbert-type inequalities related to the discrete Hardy-Littlewood-Sobolev inequality \cite[Corollary 22]{CLV} and improved bounds for the Riemann zeta-function in the critical strip under the Riemann hypothesis \cite[Theorems 1 and 2]{CC} are obtained from these results.

\medskip

In this paper we aim to build the analogous theory for truncated and odd functions.  This introduces additional difficulties due to the discontinuity at the origin. We will treat both the unrestricted $L^1(\R)$-approximation and the one-sided problem.

\medskip

The first part of the paper is devoted to the solution of the extremal problem for the truncated Gaussian $z \mapsto G_{\lambda}^{+}(z)$ defined by
\begin{align*}
G_{\lambda}^{+}(z) = 
\left\{ 
\begin{array}{cc}
G_{\lambda}(z)& {\rm for}\ \ \Re z >0,\\
1/2& {\rm for}\ \ \Re z = 0,\\
0& {\rm for}\ \ \Re z <0.
\end{array}
\right.
\end{align*}
and the odd Gaussian $z \mapsto G_{\lambda}^{o}(z)$ defined by
\begin{align*}
G_{\lambda}^{o}(z) = 
\left\{ 
\begin{array}{cc}
G_{\lambda}(z)& {\rm for}\ \ \Re z >0,\\
0& {\rm for}\ \ \Re z = 0,\\
-G_{\lambda}(z)& {\rm for}\ \ \Re z <0.
\end{array}
\right.
\end{align*}

\medskip

The second part of the paper is devoted to the integration on the free parameter $\lambda$ as a tool to generate the solution of (\ref{BS1}) for a class of truncated and odd functions. We determine the set of admissible measures for integration. The use of the Paley-Wiener theorem for distributions is classical in the unrestricted extremal problem, see for instance \cite{Ga3}, whereas in the one-sided case it has been recently used in \cite{CLV}. 

\bigskip

\section{Results}
 
Let $\tau\in\CC$ with $\Im \tau >0$. We follow the notation of Chandrasekharan \cite{chan} and define with $q= e^{\pi i\tau}$ the Jacobi theta functions 
\begin{align}
\theta_1(z,\tau) &= \sum_{n=-\infty}^\infty q^{(n+\frac{1}{2})^2} e^{(2n+1)\pi i z}\label{th1},\\
\theta_2(z,\tau) & = \sum_{n=-\infty}^\infty (-1)^n q^{n^2} e^{2\pi inz},\label{th2}\\
\theta_3(z,\tau) & = \sum_{n=-\infty}^\infty q^{n^2} e^{2\pi inz}.\label{th3}
\end{align}

Throughout this paper we define the truncation $\sgp$ by
\begin{align*}
\sgp = \tfrac12 (1+\text{sgn}(x)).
\end{align*}

Recall that the Fourier transform of the Gaussian $G_{\lambda}(x) = e^{-\pi \lambda x^2}$ is given by 
\begin{equation*}%\label{FT-Gaussian-1}
\widehat{G}_{\lambda}(t) = \int_{-\infty}^{\infty} e^{-2 \pi i t x} G_{\lambda}(x)\,\dx = \lambda^{-1/2} e^{-\pi \lambda^{-1} t^2},
\end{equation*}
and, via contour integration, the Fourier transform of the truncated Gaussian $G_{\lambda}^{+}(x) = \sgp\, e^{-\pi \lambda x^2}$ is shown to be

\begin{equation}\label{FT-Gaussian-3}
\widehat{G}_{\lambda}^{+}(t) = \frac{1}{2} \, \lambda^{-1/2} e^{-\pi \lambda^{-1} t^2} + \frac{t}{i \lambda}  \int_0^1e^{-\pi \lambda^{-1} t^2 (1-y^2)}\,\dy.
\end{equation}

\medskip

Define the following three entire functions of exponential type
\begin{align*}
K_{\lambda}^{+}(z)&=\frac{\sin\pi z}{\pi} \sum_{n=1}^\infty (-1)^{n} \left\{\frac{G_\lambda(n)}{(z-n)} -\frac{G_\lambda(n)}{z}\right\},\\
\lplus(z) &= \frac{\sin^2\pi z}{\pi^2} \sum_{n=1}^\infty \left\{\frac{G_\lambda(n)}{(z-n)^2} + \frac{G_\lambda'(n)}{z-n}-\frac{G_\lambda'(n)}{z}\right\},\\
\mplus(z) &= \frac{\sin^2\pi z}{\pi^2} \sum_{n=1}^\infty \left\{ \frac{G_\lambda(n)}{(z-n)^2} + \frac{G_\lambda'(n)}{z-n}-\frac{G_\lambda'(n)}{z}\right\} + \frac{\sin^2\pi z}{\pi^2 z^2}.
\end{align*}
Note that $K_{\lambda}^{+}$ is obtained as a function that interpolates the values of $G_{\lambda}^{+}$ at $\Z \backslash \{0\}$ with the value at $z=0$ obtained through the Poisson summation formula. In the same way, $\lplus$ and $\mplus$ interpolate the values of $G_{\lambda}^{+}$ and its derivative at $\Z \backslash \{0\}$. Note that $K_{\lambda}^{+}$ is initially defined on $\C \backslash \N$ but extends to an entire function of exponential type $\pi$, while $\lplus$ and $\mplus$, also initially defined on  $\C \backslash \N$, extend to entire functions of exponential type $2\pi$.

\medskip

The following two theorems provide the solution of the extremal problem \eqref{BS1} for the truncated Gaussian.

\begin{theorem}[Optimal two-sided approximation]\label{truncated-ba-theorem} 
The inequality
\begin{align}\label{extremal-eq}
\sin \pi x \, \big\{\gplus(x)-K_{\lambda}^{+}(x)\big\} \ge 0
\end{align}
holds for all real $x$. Let $z \mapsto K(z)$ be a function of exponential type at most $\pi$. We have
\begin{align}\label{ba-bound}
\begin{split}
\int_{-\infty}^\infty \big|\gplus(x)- K(x)\big| \, \dx \ge\, 
& \frac{1}{\pi \lambda} \int_0^1 \theta_1\left(0, i \lambda ^{-1}\big(1-y^2\big) \right) \dy\,,
\end{split}
\end{align}
with equality if and only if $K = K_{\lambda}^{+}$.
\end{theorem}
\medskip

\begin{theorem}[Optimal one-sided approximations]\label{thm-one-sided} The inequality
\begin{align}
\lplus(x) \le \gplus(x) \le \mplus(x)
\end{align}
holds for all real $x$.
\begin{itemize}
\item[(i)] Let $z\mapsto L(z)$ be a function of exponential type at most $2\pi$ which satisfies the inequality $L(x)\le \gplus(x)$ for all real $x$. Then
\begin{align}\label{minorant-value}
\int_{-\infty}^\infty \big\{\gplus(x) - L(x)\big\} \,\dx \ge  -\frac{\theta_3(0, i \lambda)}{2} + \frac12 + \frac{1}{2\sqrt{\lambda}}\,,
\end{align}
with equality if and only if $L = \lplus$.

\medskip

\item[(ii)] Let $z\mapsto M(z)$ be a function of exponential type at most $2\pi$ which satisfies the inequality $M(x)\ge \gplus(x) $ for all real $x$. Then
\begin{align}\label{majorant-value}
\int_{-\infty}^\infty \big\{M(x) - \gplus(x) \big\} \,\dx \ge \frac{\theta_3(0, i \lambda)}{2} + \frac12 - \frac{1}{2\sqrt{\lambda}}\,,
\end{align}
with equality if and only if $M = \mplus$.
\end{itemize}
\end{theorem}
\medskip

The starting point of the proofs is a decomposition of these functions into integral representations analogous to those developed in \cite{CLV}. In distinction to the representations for the Gaussian, the integrands in our representation are not Jacobi theta functions. They turn out to be certain solutions of the heat equation, and the maximum principle for the heat operator is used to obtain the necessary inequalities. The uniqueness part will follow from the interpolation properties at $\Z$. A simple dilation argument provides the optimal approximations of exponential type $2\pi\delta$ for any $\delta>0$.

\medskip

Once we have established the solution of the extremal problem \eqref{BS1} for the truncated Gaussian as described in Theorems \ref{truncated-ba-theorem} and \ref{thm-one-sided}, we can easily derive the solution of this problem for the odd Gaussian  $z \mapsto G_{\lambda}^{o}(z)$. 
Observe that
\begin{align*}
G_{\lambda}^{o}(z)  = G_{\lambda}^{+}(z) - G_{\lambda}^{+}(-z)
\end{align*}
and define the entire functions
\begin{align*}
K_{\lambda}^{o}(z) & = K_{\lambda}^{+}(z) - K_{\lambda}^{+}(-z),\\
\lodd(z) &= \lplus(z) - \mplus(-z),\\
\modd(z) &= \mplus(z) - \lplus(-z).
\end{align*}
Theorem \ref{truncated-ba-theorem} and Theorem \ref{thm-one-sided} imply
\begin{align*}%\label{extremal-eq-odd}
\begin{split}
\sin&(\pi x)  \big\{G_{\lambda}^{o}(x)-K_{\lambda}^{o}(x)\big\} \\
 &=  \sin(\pi x) \big\{G_{\lambda}^{+}(x)-K_{\lambda}^{+}(x)\big\} + \sin(-\pi x)  \big\{G_{\lambda}^{+}(-x)-K_{\lambda}^{+}(-x)\big\}  \geq 0 
\end{split}
\end{align*}
and
\begin{align*}
\lodd(x) \leq \godd(x) \leq \modd(x).
\end{align*}
These functions preserve the interpolation properties at $\Z$ and are the best approximation, extremal minorant and majorant for the odd Gaussian, respectively. This follows by arguments analogous to the proofs of Theorems \ref{truncated-ba-theorem} and \ref{thm-one-sided}, and plainly guarantees the odd counterparts of all the results we present here for truncated functions.

\medskip

Having solved the extremal problem \eqref{BS1} for a family of functions with a free parameter $\lambda >0$, we are now interested in integrating this parameter against a set of admissible non-negative Borel measures $\nu$ on $[0,\infty)$ to generate a new class of truncated (and odd) functions for which \eqref{BS1} has a solution.

\medskip

We determine in Section \ref{Asy} the set of admissible measures $\nu$. The results of that section lead us to consider non-negative Borel measures $\nu$ on $[0,\infty)$ satisfying one of the conditions 
\begin{align} 
&\int_0^{\infty}  \frac{1}{1 + \sqrt{\lambda}}\, \dnu <\infty\,,\label{nu-1}\\
&\int_0^{\infty}   \dnu <\infty\,.\label{nu-2}
\end{align}

We define the truncated function $g:\R \to \R$ given by 
\begin{equation*}%\label{class}
g(x) = \sgp \, \int_0^{\infty}  e^{-\pi \lambda x^2} \, \dnu\,.
\end{equation*}

\begin{theorem}[Optimal two-sided approximation - general case]\label{thm-app}
Let $\nu$ satisfy \eqref{nu-1}. Then there exists a unique best approximation $z\mapsto k(z)$ of exponential type $\pi$ for $x\mapsto g(x)$. The function $k$ interpolates the values of $g$ at $\Z\backslash\{0\}$, satisfies
\begin{equation*}
\sin \pi x \, \{ g(x) - k(x) \} \geq 0
\end{equation*}
and 
\begin{equation*}
\int_{-\infty}^{\infty} |g(x) - k(x)| \, \dx   = \int_0^{\infty}  \frac{1}{\pi \lambda} \int_0^1 \theta_1\left(0, i \lambda ^{-1}\big(1-y^2\big) \right) \dy\,\dnu.
\end{equation*}
\end{theorem}

\begin{theorem}[Optimal one-sided approximations - general case]\label{thm-app-one-sided}~
\medskip
\begin{itemize}
\item [(i)] {\rm  (Extremal minorant) } Let $\nu$ satisfy \eqref{nu-1}. Then there exists a unique extremal minorant $z\mapsto l(z)$ of exponential type $2\pi$ for $x\mapsto g(x)$. The function $l$ interpolates the values of $g$ and its derivative at $\Z\backslash\{0\}$ and satisfies
\begin{equation*}
\int_{-\infty}^{\infty} \{g(x) - l(x)\}\, \dx   = \int_0^{\infty} \left\{ - \frac{\theta_3(0, i \lambda)}{2} + \frac12 + \frac{1}{2\sqrt{\lambda}}\right\}\, \dnu.
\end{equation*}
\item[(ii)] {\rm (Extremal majorant)} Let $\nu$ satisfy \eqref{nu-2}. Then there exists a unique extremal majorant $z\mapsto m(z)$ of exponential type $2\pi$ for $x\mapsto g(x)$. The function $m$ interpolates the values of $g$ and its derivative at $\Z\backslash\{0\}$ and satisfies
\begin{equation*}
\int_{-\infty}^{\infty} \{m(x) - g(x)\}\, \dx   = \int_0^{\infty} \left\{ \frac{\theta_3(0, i \lambda)}{2} + \frac12 - \frac{1}{2\sqrt{\lambda}}\right\}\, \dnu.
\end{equation*}
\end{itemize}
\end{theorem}

Let us mention one immediate application of the results obtained here to the theory of the Riemann zeta-function. When one considers the odd Gaussian $G_{\lambda}^{o}(x) = \sgn(x) e^{-\pi \lambda x^2}$ integrated against the {\it finite} measure
\begin{equation*}%\label{measure00}
\dnu = \left\{\int_0^{\infty} \frac{t}{2\sqrt{\pi \lambda^3}} \,e^{-\tfrac{t^2}{4\lambda}} \left( \frac{1}{t} \sin (\sqrt{\pi}t) - \sqrt{\pi} \cos (\sqrt{\pi} t)\right) \dt\right\}\dl\,,
\end{equation*}
one arrives at the odd function
\begin{align*}
g(x) =  \arctan\left(\frac{1}{x}\right) - \frac{x}{1+x^2}.
\end{align*}

This was observed by Carneiro, Chandee and Milinovich in \cite{CCM} and they used the extremals for this function to improve the upper bound for the argument function $S(t)$ on the critical line under the Riemann hypothesis. The previous best bound had been obtained by Goldston and Gonek \cite{GG}.

\medskip

Further applications include the odd and truncated counterparts of the applications in \cite{CLV}, we refer to that paper for the definition of measures.

\section{Integral representations for the Gaussian}%%%%%%%%%%%%%%%%%%%%%%%%%%%%%%%%%%%%%%%%%%%%%%%%%%%%%%%%%%%%%%%%%%%%%%%%%%%%%

Let $\lambda>0$ and recall that $G_\lambda(z)  = e^{-\pi\lambda z^2}$. We note that
\begin{align}\label{repginv}
G_{\lambda}(z)^{-1} = \lambda^{\frac12} \int_{-\infty}^\infty e^{-2\pi \lambda z u} \,G_{\lambda}(u) \,\du
\end{align}
for all complex $z$. In this section we collect auxiliary integral representations and estimates for the Gaussian. It was shown in \cite{CLV} that for distinct $w,z\in\CC$ and $\lambda>0$ the identity
\begin{align}\label{intrep2-eq}
\begin{split}
\frac{G_{\lambda}(z) - G_{\lambda}(w)}{z - w}
	&= 2\pi \lambda^{\frac32} \int_{-\infty}^0 \int_{-\infty}^0 e^{-2\pi\lambda tu} \, G_{\lambda}(z - t)\, G_{\lambda}(w-u)\,\du\, \dt \\
        &\qquad - 2\pi \lambda^{\frac32} \int_0^{\infty} \int_0^{\infty} e^{-2\pi\lambda tu} \, G_{\lambda}(z - t)\, G_{\lambda}(w-u) \,\du \,\dt.
\end{split}
\end{align}
is valid. We require similar representations for $(z-w)^{-1}G(w)$ and $(z - w)^{-1} G(z)$.

\begin{lemma}\label{intrep} Let $\lambda>0$, and $w,z\in\CC$.  For $\Re z<\Re w$ we have the identities
\begin{align}
\begin{split}
\frac{G_\lambda(w)}{z-w} \label{intrep-1}
&=-2\pi\lambda^{\frac32} \int_{-\infty}^\infty \int_{-\infty}^0 e^{-2\pi \lambda tu}\, G_\lambda(z-t) \, G_\lambda(w-u)\,\du\, \dt
\end{split}
\end{align}
and
\begin{align}
\begin{split}\label{gauss2}
\frac{G_{\lambda}(z)}{z - w}
	&= -2\pi \lambda^{\frac32} \int_{0}^\infty \int_{-\infty}^0 e^{-2\pi\lambda tu} \,G_{\lambda}(z - t)\,G_{\lambda}(w-u)\,\du\, \dt \\
        &\qquad - 2\pi \lambda^{\frac32} \int_0^{\infty} \int_0^{\infty} e^{-2\pi\lambda tu} \,G_{\lambda}(z - t)\,G_{\lambda}(w-u) \,\du \,\dt\,,\\
\end{split}
\end{align}
while for $\Re z>\Re w$ we have the identities
\begin{align}
\frac{G_\lambda(w)}{z-w} &= \label{intrep-2}
\displaystyle 2\pi \lambda^{\frac32}\int_{-\infty}^\infty \int_{0}^\infty e^{-2\pi \lambda tu}  \, G_\lambda(z-t)  \,  G_\lambda(w-u) \,\du \,
\end{align}
and
\begin{align}
\begin{split}
\frac{G_{\lambda}(z)}{z - w} \label{gauss3}
	&= 2\pi \lambda^{\frac32} \int_{-\infty}^0 \int_{-\infty}^0 e^{-2\pi\lambda tu} \,G_{\lambda}(z - t)\,G_{\lambda}(w-u)\,\du \,\dt \\
        &\qquad + 2\pi \lambda^{\frac32} \int_{-\infty}^0 \int_0^{\infty} e^{-2\pi\lambda tu} \,G_{\lambda}(z - t)\,G_{\lambda}(w-u) \,\du \,\dt.
\end{split}
\end{align}
\end{lemma} 

\begin{proof} To prove \eqref{intrep-1}, let $\lambda=1$ and set $G(z):=G_1(z)$. We have for $\Re z< \Re w$
\begin{align}\label{rep-zrecip}
\frac{1}{2\pi(z-w)} = -\int_{-\infty}^0 e^{-2\pi z t} \, e^{2\pi wt} \,\dt =\int_{-\infty}^\infty e^{-2\pi z t} \,h_w(t) \,\dt\,,
\end{align}
where $h_w(t) =0$ if $t>0$, and $h_w(t) =-e^{2\pi wt}$ if $t\le 0$.  From \eqref{repginv} and \eqref{rep-zrecip} it follows that the product of $G(z)^{-1}$ and $[2\pi(z-w)]^{-1}$ is the two-sided Laplace transform of the integral convolution of $G$ with $h_w$, and this Laplace transform converges absolutely for $\Re z<\Re w$. Hence we obtain 
\begin{align}\label{irep1-1}
\frac{1}{z-w} = -2\pi G(z) \int_{-\infty}^\infty e^{-2\pi zt}\int_{-\infty}^0 e^{2\pi wu} \,G(t-u) \,\du\,\dt
\end{align}
for $\Re z<\Re w$. We multiply \eqref{irep1-1} by $G(w)$ and use the identity
\[
G(z) \, G(w) \, e^{2\pi (wu-tz)} \, G(t-u) = G(z+t) \, e^{2\pi tu} \, G(w-u)
\]
to get \eqref{intrep-1} for $\lambda=1$. A change of variable gives the result for arbitrary $\lambda$. Addition of \eqref{intrep2-eq} and \eqref{intrep-1} gives \eqref{gauss2}. The identities for $\Re z > \Re w$ are shown with an analogous argument.
\end{proof}

\section{Auxiliary inequalities}\label{A}%%%%%%%%%%%%%%%%%%%%%%%%%%%%%%%%%%%%%%%%%%%%%%%%%%%%%%%%%%%%%%%%%%%%%%%%%%%%%%%
The classical Jacobi theta functions play an important role in the extremal problem for the truncated Gaussian. This section collects inequalities for these functions and related expressions that are needed in the later sections. For easy reference we list the product representations (cf.\ \cite[Chapter V, Theorem 6]{chan}) 
\begin{align}
\theta_1(z,\tau) &= q^{1/4} \,e^{\pi i z}\,\prod_{n=1}^\infty\big(1- q^{2n}\big)\prod_{n=1}^\infty\big(1+ q^{2n} e^{2\pi iz}\big) \prod_{n=1}^\infty \big(1+ q^{2n-2} e^{-2\pi iz}\big),\label{th1-prod-1}\\
\theta_2(z,\tau) &= \prod_{n=1}^\infty\big(1-q^{2n}\big)\prod_{n=1}^\infty\big(1-q^{2n-1} e^{2\pi iz}\big) \prod_{n=1}^\infty \big(1-q^{2n-1} e^{-2\pi iz}\big),\label{th2-prod}\\
\theta_3(z,\tau) &= \prod_{n=1}^\infty\big(1-q^{2n}\big)\prod_{n=1}^\infty\big(1+q^{2n-1} e^{2\pi iz}\big) \prod_{n=1}^\infty \big(1+q^{2n-1} e^{-2\pi iz}\big),\label{th3-prod}
\end{align}
and the transformation formulas (cf.\ \cite[Chapter V, Theorem 9 and Corollary 1]{chan})
\begin{align}
\lambda^{-\frac12} \theta_1\big(z,i\lambda^{-1}\big) & = \sum_{n=-\infty}^\infty (-1)^n G_\lambda(z-n),\label{th1-series}\\
\lambda^{-\frac12} \theta_2\big(z,i\lambda^{-1}\big) & = \sum_{n=-\infty}^\infty G_\lambda\big(z-n-\hh\big),\label{th2-series}\\
\lambda^{-\frac12} \theta_3\big(z,i\lambda^{-1}\big) & = \sum_{n=-\infty}^\infty G_\lambda\big(z-n\big),\label{th3-series}
\end{align}
for all $z\in\CC$. In the following we let $\theta_j'(z,\tau) = \frac{{\rm d}}{\dz}\theta_j(z,\tau)$, where $j=1,2,3$.
 
\begin{lemma}\label{theta2} 
Let $\lambda>0$.  Then $i\,\theta_2'(ix,i\lambda) >0$ for all real $x$ with $- \frac{\lambda}{2} <x<0$, and $\theta_3'(x,i\lambda)\le 0$ for all real $x$ with $0\le x\le 1/2$.
\end{lemma}

\begin{proof} The series representation (cf.\ \cite[page 489]{ww})
\begin{align*}
\frac{\theta_2'(z,\tau)}{\theta_2(z,\tau)} = 4\pi  \sum_{n=1}^\infty \frac{q^n \sin 2\pi nz}{1-q^{2n}} \qquad \Big(|\Im z|<  \frac{\Im\tau}{2}\Big),
\end{align*}
gives with $z=ix$, $\tau=i\lambda$ and $\sin(2\pi nz) = i\sinh(2\pi nx)$ 
\begin{align*}
i\theta_2'(ix,i\lambda) = -2 \pi\,\theta_2(ix,i\lambda) \sum_{n=1}^\infty \frac{\sinh(2\pi nx)}{\sinh(\pi n\lambda)} \qquad(|x|<\lambda/2).
\end{align*}
Since for $-\lambda/2<x\le 0$ the function $x\mapsto \theta_2(ix,i\lambda)$ is positive (from \eqref{th2-prod}), the first inequality of the lemma follows. The inequality for $\theta_3'$ follows similarly from the representation
\[
\frac{\theta_3'(z,\tau)}{\theta_3(z,\tau)} = -4\pi\sin(2\pi z)\sum_{n=1}^\infty \frac{q^{2n-1}}{1+2q^{2n-1} \cos(2\pi z) + q^{4n-2}}.
\]
\end{proof}

\begin{lemma}\label{ineqTheta-1}
Let $\lambda>0$. For $0\le x< 1/2$ we have
\begin{align}\label{l1-th1}
\frac{\theta_1\big(x,i\lambda^{-1}\big)}{\theta_1\big(0,i\lambda ^{-1}\big)} \le  G_{\lambda}(x).
\end{align}
\end{lemma}

\begin{proof} The definition  \eqref{th2} and the transformation formula \eqref{th1-series} give
\begin{align}\label{th1toth2}
\lambda^{-\frac{1}{2}} \theta_1\big(x,i\lambda^{-1}\big) = G_{\lambda}(x) \, \theta_2(-i\lambda x, i\lambda).
\end{align}
In particular $\lambda^{-\frac{1}{2}} \theta_1\big(0,i\lambda^{-1}\big) = \theta_2\big(0, i\lambda\big)$ which is positive. Lemma \ref{theta2} implies that
\begin{align*}
\frac{{\rm d}}{\dx}\theta_2(-i\lambda x,i\lambda) = -i\,\lambda\,\theta_2'(-i\lambda x,i\lambda) <0\,,
\end{align*}
and hence
\begin{align}
\theta_2(-i\lambda x,i\lambda) \le \theta_2(0,i\lambda)\label{derivative}
\end{align}
for $0\le x<1/2$. Identity \eqref{th1toth2} and inequality \eqref{derivative} imply \eqref{l1-th1}.
\end{proof}

\begin{lemma} \label{ineq-th-2}
Let $\lambda >0$. For $x >0$ we have
\begin{align}\label{needed-for-l2l3}
\int_0^\infty e^{-2\pi \lambda xt} \big\{\theta_1\big(t,i\lambda^{-1}\big) - \theta_1\big(0,i\lambda^{-1}\big) G_\lambda(t)\big\}\,\dt < 0.
\end{align}
\end{lemma}

\begin{proof} By Lemma \ref{ineqTheta-1} we obtain  
\begin{align}\label{0to12}
\int_0^{1/2} e^{-2\pi\lambda xt}  \big\{\theta_1\big(t,i\lambda^{-1}\big) - \theta_1\big(0,i\lambda^{-1}\big) G_\lambda(t)\big\}\,\dt \le0\,,
\end{align}
and the integrand is not identically zero, so the inequality is strict. The  identity $\theta_1(z+1,\tau) = -\theta_1(z,\tau)$  implies
\begin{align}\label{12toinfty}
\begin{split}
\int_{1/2}^\infty e^{-2\pi \lambda xt} \, \theta_1\big(t,i\lambda^{-1}\big)\,\dt &= \sum_{n=0}^\infty (-1)^n \int_{0}^{1} e^{-2\pi \lambda x(t+\frac12+n)}\,  \theta_1\big(t+\tfrac12,i\lambda^{-1}\big)\,\dt\\
&= \frac{1}{\big(e^{\pi\lambda x}+e^{-\pi\lambda x}\big)}\int_0^1  e^{-2\pi \lambda x t } \, \theta_1\big(t+\tfrac12,i\lambda^{-1}\big)\,\dt\,,
\end{split}
\end{align}
and since $\theta_1\big(t,i\lambda^{-1}\big) \le 0$ for $\tfrac12\le t\le \tfrac32$, we see that \eqref{0to12} and \eqref{12toinfty} imply \eqref{needed-for-l2l3}.
\end{proof}

The remaining part of this section is devoted to the proof of the following three inequalities.

\begin{proposition} For $t>0$ we have
\begin{align}
\sum_{n=1}^\infty (-1)^{n+1} n^2 e^{-tn^2} &\ge 0,\label{partialq}\\
\sum_{n=0}^\infty e^{-tn^2} \big(1-2tn^2\big)&\ge \frac12,\label{ineq-20}\\
\sum_{n=1}^\infty e^{-tn^2} \big(tn^3- n\big)&\ge 0.\label{ineq-31}
\end{align}
\end{proposition}

\begin{proof}[Proof of \eqref{partialq}] Consider
\begin{align*}%\label{periodic-sol}
\theta_2(0,it) = 1+ 2 \sum_{n=1}^\infty (-1)^n e^{-\pi tn^2}
\end{align*}
for positive $t$. The product formula \eqref{th2-prod} implies
\begin{align}\label{monotone}
\begin{split}
\theta_2(0,it)&= \prod_{n=1}^\infty\big(1-e^{-2\pi nt}\big)\prod_{n=1}^\infty\big(1-e^{-(2n-1)\pi t}\big)^2.
\end{split}
\end{align}
Every factor in the products of \eqref{monotone} is a positive, monotonically increasing function of $t$ for $t>0$, hence $\partial_t \theta_2(0,it)\ge 0$ which implies \eqref{partialq}.
\end{proof}

\begin{proof}[Proof of \eqref{ineq-20}] Differentiation of the functional equation (from \eqref{th3-series})
\begin{align*}%\label{fueq-theta}
\pi^{\frac12}\sum_{n=-\infty}^\infty e^{-\frac{\pi^2 n^2}{t}} = t^{\frac12} \sum_{n=-\infty}^\infty e^{-tn^2},
\end{align*}
gives
\begin{align*}%\label{fueq-theta-1}
\begin{split}
\pi^{\frac52} t^{-2} \sum_{n=-\infty}^\infty n^2 e^{-\frac{\pi^2 n^2}{t}} &= \frac{1}{2t^{\frac12}}  \left(1 +2\sum_{n=1}^\infty e^{-tn^2} (1-2tn^2)\right),
\end{split}
\end{align*}
and the left-hand side is non-negative, which gives \eqref{ineq-20}.
\end{proof}

The proof of \eqref{ineq-31} for $0<t<1$ is surprisingly involved, since there does not appear to be an identity analogous to the functional equation for $\sum_n |n| e^{-tn^2}$. We define \textsl{Dawson's integral} $x\mapsto D(x)$ by
\begin{align*}
D(x) =  \int_0^x e^{(u^2-x^2)}\, \du\,.
\end{align*}

We require lower bounds by rational functions and an integral evaluation involving $D$.

\begin{lemma} We have
\begin{align}\label{dawson1}
D(x)\ge \frac{1}{2x} \ \ \text{ for }\  \ x\ge 1\,,
\end{align}
and
\begin{align}\label{dawson2}
D(x)\ge \frac{x^2-1}{x(2x^2-3)} \ \ \text{ for }\ \  x\ge 2.
\end{align}
\end{lemma}

\begin{proof} We define $x\mapsto g(x)$ for $x>0$ by
\[
g(x) = \int_0^x e^{u^2} \du - \frac{e^{x^2}}{2x}.
\]
Since $D(1) >\frac12$ and $g'(x) = (2x^2)^{-1} \exp(x^2)>0$ for all positive $x$, we obtain $g(x)\geq0$ for $x\ge1$ and hence \eqref{dawson1}. The inequality $D(2)> 3/10$ and differentiation of
\[
x\mapsto  \int_0^x e^{u^2} \du - e^{x^2} \frac{x^2-1}{x(2x^2-3)}
\]
gives \eqref{dawson2} with an analogous argument.
\end{proof}

\begin{lemma} \label{hypergeom-id}
We have for real $x$ and $j=0,2,4$
\begin{align*}
\begin{split}
\int_0^\infty u^j e^{-u^2} \sin(2xu) \,\du  = \begin{cases}
 D(x)&\text{ if }j=0,\\
 \\
\frac12 x+ D(x) \big(\frac12 - x^2\big)&\text{ if }j=2,\\
\\
\frac54 x -\frac12 x^3 + D(x) \big( \frac34 - 3x^2 +x^4\big)&\text{ if }j=4.
\end{cases}
\end{split}
\end{align*}
\end{lemma}

\begin{proof}
The case $j=0$ follows by contour integration
\begin{align*}
\int_0^\infty e^{-u^2} \sin(2xu) \,\du & = \int_0^\infty e^{-u^2} \left( \frac{e^{2 i x u} - e^{-2 i x u}}{2i}\right)\,\du\\
& = e^{-x^2} \int_0^{\infty}\frac{e^{-(u-ix)^2}}{2i}\,\du \ - \ e^{-x^2} \int_0^{\infty}\frac{e^{-(u+ix)^2}}{2i}\,\du \\
& = e^{-x^2} \int_0^x e^{u^2} \du = D(x).
\end{align*}
The steps $j \mapsto j+2$ follow by repeated integration by parts.
\end{proof}

\begin{proof}[Proof of \eqref{ineq-31}] Inequality \eqref{ineq-31} is trivially true for $t\ge 1$. In order to prove it for $0<t<1$ we define the  periodic Bernoulli function $u\mapsto B(u)$ by
\begin{align*}
B(u) = u - [u]-\frac{1}{2}\,,
\end{align*}
where $[u]$ is the greatest integer $\le u$. We define furthermore for $j\in\NN$ and $t>0$
\begin{align}\label{def-ni}
t\mapsto N_j(t) = \int_{0^-}^\infty u^j e^{-tu^2} \textrm{d}B(u)
\end{align}
and note that
\begin{align}\label{ni-to-separate}
N_j(t) = \int_0^\infty u^j e^{-tu^2} \du \ - \ \sum_{n=0}^\infty n^j e^{-tn^2}.
\end{align}
We perform an integration by parts in \eqref{def-ni} and replace $B$ by its Fourier series expansion. Since the partial sums of the Fourier series of $B$ are uniformly bounded \cite[Vol I, page 61]{Z}, Lebesgue dominated convergence followed by an application of Lemma \ref{hypergeom-id} gives, for $j\ge 1$, with $x_{n,t} = t^{-\frac12} n\pi$
\begin{align*}
\begin{split}
N_j(t) &= -\int_0^\infty B(u) \, e^{-tu^2} \, u^{j-1} \, \big(j-2tu^2\big) \,\du\\
&=\sum_{n=1}^\infty \frac{1}{\pi n} \int_0^\infty   e^{-tu^2} \,u^{j-1} \,\big(j-2tu^2\big)\,\sin(2\pi nu) \,\du\\
&=\begin{cases}
\displaystyle t^{-1} \sum_{n=1}^\infty \{2 x_{n,t} D(x_{n,t}) - 1\}&(j=1),\\
\displaystyle t^{-2} \sum_{n=1}^\infty \left\{x_{n,t}^2 -1+x_{n,t}\big(3-2x_{n,t}^2\big)D(x_{n,t}) \right\}&(j=3).
\end{cases}
\end{split}
\end{align*}
Since $x_{n,t} = t^{-\frac12} \pi n\ge \pi$ for $0<t<1$ and $n\ge 1$, inequality \eqref{dawson1} implies that $N_1(t)\ge 0$ and inequality \eqref{dawson2} implies that $N_3(t) \le 0$  for all $0<t<1$. Inserting this into \eqref{ni-to-separate} gives for $0<t<1$
\begin{align*}
\sum_{n=1}^\infty n e^{-tn^2} \le \int_0^\infty ue^{-tu^2} \,\du = t\int_{0}^\infty u^3 e^{-tu^2} \,\du \le t \sum_{n=1}^\infty n^3 e^{-tn^2}
\end{align*}
which finishes the proof of \eqref{ineq-31}.
\end{proof}

\section{Estimates for truncated theta series}%%%%%%%%%%%%%%%%%%%%%%%%%%%%%%%%%%%%%%%%%%%%%%%%%%%%%%%%%%%%%%%%%%%%%%%%%

Let $\lambda >0$. We will be working with the truncations $z\mapsto \thp(z,\lambda)$ and $z\mapsto \vartheta^{+}(z,\lambda)$  defined by
\begin{align}\label{vth+Intro}
\begin{split}
\thp(z,\lambda)& = \sum_{n=1}^\infty (-1)^{n+1} G_\lambda(z-n)\\ 
\vartheta^{+}(z,\lambda)& =  2\pi \lambda \sum_{n=1}^\infty (n-z) G_\lambda(n-z).
\end{split}
\end{align}
We denote their partial sums by $\thpn$ and $\vthpn$, respectively.  By \eqref{th1-series} and \eqref{th3-series} note that
\begin{align}
-\lambda^{-\frac12} \theta_1\big(z,i\lambda^{-1}\big) &= \thp(z,\lambda) + \thp(-z,\lambda) - G_\lambda(z),\label{thp-to-th1} \\
\lambda^{-\frac12} \theta'_3\big(z,i\lambda^{-1}\big) &= \vthp(z,\lambda) - \vthp(-z,\lambda) + G'_\lambda(z).\label{vthp-to-th3}
\end{align}
For easy reference we collect some growth estimates. The proofs are straightforward, and we omit them.

\begin{lemma}\label{ub-gauss}
Let $\lambda>0$. Then
\begin{align}
|\thpn(u,\lambda)| & \le 2 G_\lambda(u-1) & (u\le 1), \label{thp-growth1}\\
|\thpn(u,\lambda)| & = \mathcal{O}(1)&(u\ge 1),\label{thp-growth2}
\end{align}
where the implied constant is independent of $N$. Moreover, for $u\le 0$
\begin{align}\label{II-theta-N-ub}
0\le \vthp_N(u,\lambda) \le \vthp(u,\lambda) \le c_\lambda (|u|+1) G_\lambda(1-u)
\end{align}
holds with $c_\lambda = 2\pi \lambda \,G_\lambda(1)^{-1} \sum_{n\ge 1} nG_\lambda(n)$ independent of $u$.
\end{lemma}

We develop estimates for truncated theta functions that will be crucial for the proofs of the main theorems. We accomplish this via the maximum principle for the heat operator, primarily a tool in partial differential equations. We denote by $\partial_x f$ the partial derivative of $f$ with respect to $x$ and by $L$  the partial differential operator acting on $(x,t)\mapsto f(x,t)$ given by
\begin{align*}
f\mapsto L[f] = \partial_{xx}f - \partial_t f.
\end{align*}
For an open rectangle $E$ we let
\begin{align*}
C^{2,1}(E) = \{f \in C^1(E)\,|\, \partial_{xx}f(x,t)\in C(E)\}.
\end{align*}
The maximum principle for the heat operator can be found for instance in \cite[Chapter 3]{PW}. It essentially states that for a function $f(x,t)$ with $L[f]\ge 0$, knowledge of $f$ on the boundary of a rectangle implies information about $f$ inside the rectangle.

\begin{lemma}[Maximum principle for the heat equation]\label{maxprinc} 
Let $\ell>0$ and $T>0$. Consider the open rectangle $E = (-\ell,0)\times (0,T)$. Assume that $f \in C^{2,1}((-\infty,0)\times (0,\infty))$ and 
\begin{align*}
L[f]\ge 0
\end{align*}
on $E$. Let $F$ be the closed set given by the union of bottom, left side, and right side of the closure of $E$ (i.e. $F = \{-\ell\} \times [0,T] \cup [-\ell,0] \times \{0\} \cup \{0\} \times [0,T]$). If
\begin{align*}
\limsup_{(x,t)\in E\to (x_0,t_0)\in F} f(x,t) \le M
\end{align*}
for all $(x_0,t_0)\in F$, then
\begin{align*}
f(x,t)\le M
\end{align*}
for all $(x,t)\in E$.
\end{lemma}

\begin{proof} We refer to \cite[Chapter 3]{PW}. The lemma is shown there under the stronger assumption that $f$ is continuous on the closure of $E$; the statement given here follows with an approximation from inside.
\end{proof}

Throughout this section we will be working with the function $(x,t) \mapsto k(x,t)$ defined for $(x,t) \in \R \times (0,\infty)$ by
\begin{align}\label{def-k}
k(x,t) = t^{-1/2} e^{-\frac{x^2}{4t}}.
\end{align}
Note that $k$ solves the heat equation $L[k]=0$ at every $(x,t)\in \RR\times (0,\infty)$.

\begin{proposition}
Let $x<0$ and $\lambda>0$. Then
\begin{align}
 \thp(0,\lambda) \,G_\lambda(x) - \thp(x,\lambda) &\geq 0,\label{thp-ineq-eq}\\
\vthp(0,\lambda) \,G_\lambda(x) - \vthp(x,\lambda) &\ge 0, \label{vartheta-ineq1}\\
\vthp(0,\lambda) \, G_\lambda(x) - \vthp(x,\lambda) &\le \frac{G_\lambda'(x)}{2}. \label{vartheta2}
\end{align}
Moreover, for $0\le x\le 1/2$ and $\lambda>0$
\begin{align}\label{vartheta3}
\vthp(0,\lambda) G_\lambda(x)\le \vthp(0,\lambda) \le \vthp(x,\lambda).
\end{align}
\end{proposition}

\begin{proof} The arguments are similar for all the inequalities. We give full details for the proof of \eqref{thp-ineq-eq}, and omit some of the details for the other inequalities.

\medskip

To prove \eqref{thp-ineq-eq}, define the function $(x,t) \mapsto u(x,t)$  by
\begin{align*}
u(x,t) = t^{-1/2} \,\thp\big(x,(4\pi t)^{-1}\big) = \sum_{n=1}^\infty (-1)^{n+1} k(x-n,t),
\end{align*}
and consider 
\begin{align*}
f(x,t) = u(x,t) - k(x,t)\sum_{n=1}^\infty (-1)^{n+1} e^{-\frac{n^2}{4t}}.
\end{align*}
With the change of variable $\lambda = (4 \pi t)^{-1}$ note that \eqref{thp-ineq-eq} is equivalent to 
$$f(x,t)\leq 0$$
for all $x<0$ and $t>0$. We check the assumptions of Lemma \ref{maxprinc}. 

\begin{enumerate}
\item[(i)] $L[f]\ge 0$: Since $u_{xx} - u_t = k_{xx}-k_t =0$, we obtain for all $x<0$ and $t>0$, using \eqref{partialq}, that
\begin{align*}
L[f](x,t)&= \frac{k(x,t)}{4 t^2 }\sum_{n=1}^\infty (-1)^{n+1} n^2 e^{-\frac{n^2}{4t}}\ge 0.
\end{align*}

\item[(ii)] $\{(x,0):x<0\}$ and $\{(0,t):t>0\}$: Since $k(0,t)=t^{-1/2}$ for $t>0$ and $k(x,0) =0$ for $x<0$ we obtain $f(0,t)= 0$ and $f(x,0) =0$. 
\item[(iii)] Limsup at $(0,0)$: The function $u$ is continuous at the origin for $t\to 0+$ and $u(0,0) =0$. Since for any $t>0$
\[
\sum_{n=1}^\infty (-1)^{n+1} e^{-\frac{n^2}{4t}} >0,
\]
we have
\[
\limsup_{\substack{(x,t)\to(0,0)\\x< 0,t> 0}} f(x,t) \le 0.
\]

\item[(iv)] $\{(-\ell, t):t>0\}$ with sufficiently large $\ell$:  Let $\varepsilon>0$. It can be checked directly that for all $\ell\ge \ell_0(\varepsilon)>0$ we have
\[
f(-\ell, t) \le \varepsilon.
\]
\end{enumerate}
An application of Lemma \ref{maxprinc} implies that $f(x,t)\le 0$ for $x<0$ and $t>0$, finishing the proof of \eqref{thp-ineq-eq}.

\medskip

For the proof of \eqref{vartheta-ineq1} we note that the substitution $\lambda = (4\pi t)^{-1}$ gives
\begin{align}\label{rewrite-theta-heat}
\vthp(x,\lambda) &=2\pi\lambda \sum_{n=1}^\infty (n-x)\, e^{-\pi\lambda(n-x)^2}= \frac{1}{2 t} \sum_{n=1}^\infty (n-x) \,e^{-\frac{(n-x)^2}{4t}}.
\end{align}
Define $(x,t)\mapsto f(x,t)$ by
\begin{align*}%\label{def-of-f}
f(x,t) &= t^{-\frac12}\, \vthp\big(0,(4\pi t)^{-1}\big) G_{(4\pi t)^{-1}}(x) - t^{-\frac12} \,\vthp\big(x,(4\pi t)^{-1}\big)
\end{align*}
and note that \eqref{rewrite-theta-heat} implies
\begin{align*}%\label{rep-of-f}
\begin{split}
f(x,t) &= t^{-\frac12} \frac{1}{2t} \sum_{n=1}^\infty n \,e^{-\frac{n^2}{4t}} \,e^{-\frac{x^2}{4t}} - \frac{1}{2t^{\frac32}} \sum_{n=1}^\infty (n-x)\,e^{-\frac{(n-x)^2}{4t}}\\
&= k(x,t) \sum_{n=1}^\infty \frac{n}{2t} \,e^{-\frac{n^2}{4t}} -\sum_{n=1}^\infty k_x(x-n,t).
\end{split}
\end{align*}
Apply $L=\partial_{xx} -\partial_t$. The function $f(x,t)$ is continuous on the quadrant $x\le 0$ and $t\ge 0$ with the possible exception of $(0,0)$. Since $k$ and $k_x$ are in the kernel of $L$, an application of \eqref{ineq-31} after the substitution $t\mapsto (4t)^{-1}$ gives $L[f](x,t) \le 0$ for $t>0$. We apply Lemma \ref{maxprinc} to $-f$. The required inequalities on the sets $\{(0,t):t>0\}$, $\{(x,0):x<0\}$, $\{(L,t):t>0\}$ with $L<0$ and sufficiently large $|L|$, and at the origin can be checked directly, the calculations are omitted. This finishes the proof of \eqref{vartheta-ineq1}.

\medskip

The proof of \eqref{vartheta2} requires some preparation. We define $(x,\lambda)\mapsto h(x,\lambda)$ by
\begin{align}\label{def-of-h}
h(x,\lambda) = \sum_{n=1}^\infty n e^{-\pi\lambda n^2} +\frac{x}{2}+\sum_{n=1}^\infty (x-n) \,e^{\pi\lambda (2nx -n^2)} 
\end{align}
and we note from \eqref{vth+Intro} that
\begin{align*}
2\pi\lambda &\,e^{-\pi\lambda x^2}\,h(x,\lambda)\\
 &= 2\pi\lambda \sum_{n=1}^\infty n e^{-\pi\lambda(n^2+x^2)}+\pi\lambda x e^{-\pi\lambda x^2}- 2\pi\lambda \sum_{n=1}^\infty (n-x) e^{-\pi\lambda(n-x)^2}  \\
&=\vthp(0,\lambda) G_\lambda(x)- \frac{G_\lambda'(x)}{2} - \vthp(x,\lambda),
\end{align*}
hence we need to show that
\begin{align*}%\label{to-show-ineq}
h(x,\lambda)\le 0
\end{align*}
for $x<0$ and $\lambda>0$. Since \eqref{def-of-h} implies
\begin{align*}%\label{value-h-0}
h(0,\lambda) = 0
\end{align*}
for $\lambda>0$, it suffices to show that $x\mapsto h(x,\lambda)$ is an increasing function on $(-\infty,0)$ for every $\lambda > 0$. We have from \eqref{def-of-h}
\begin{align*}
h_x(x,\lambda) &= \frac12 + \sum_{n=1}^\infty \big(1+2\pi\lambda n (x-n)\big)e^{-\pi\lambda(n^2-2nx)}.
\end{align*}
We let $\lambda = (4\pi t)^{-1}$ and  define $g$ by 
\begin{align*}
 g(x,t) = k(x,t) \,h_x\big(x,(4\pi t)^{-1}\big),
\end{align*}
where $k$ is given by \eqref{def-k}. We shall apply Lemma \ref{maxprinc} to $-g$. Note that
\begin{align*}
\begin{split}
g(x,t) &= \frac12 k(x,t) + \sum_{n=1}^\infty \left( t^{-\frac12} \,e^{-\frac{(x-n)^2}{4t}} - \frac{n}{2 t^{\frac32}} (n-x) \,e^{-\frac{(x-n)^2}{4t}}  \right)\\
&= \frac12 k(x,t) + \sum_{n=1}^\infty k(x-n,t) - \sum_{n=1}^\infty n k_x(x-n,t)\,,
\end{split}
\end{align*}
and in particular $L[g] =0$. For the boundary calculations we obtain, using \eqref{ineq-20} with $t\mapsto (4t)^{-1}$,
\begin{align*}
g(0,t) &= \frac{1}{\sqrt{t}} \left(\frac12+ \sum_{n=1}^\infty e^{-\frac{n^2}{4t}} \left(1- \frac{n^2}{2t}\right)\right)\ge  0\,,
\end{align*}
which gives the required inequality on $\{(0,t):T>0\}$. The remaining sides can be checked directly, proofs are omitted. An application of Lemma \ref{maxprinc} finishes the proof of \eqref{vartheta2}.

\medskip

The left side of \eqref{vartheta3} follows immediately from $\vthp(0,\lambda)>0$. In order to show the second inequality in \eqref{vartheta3} we establish that $x\mapsto \vthp(x,\lambda)$ is an increasing function on $[0,1/2]$ for $\lambda>0$. The goal is therefore to show for the partial derivative $\vthp_x$ that
\begin{align}\label{vthp-increasing}
\vthp_x(x,\lambda)\ge 0
\end{align}
for $0\le x\le 1/2$ and $\lambda>0$. We apply the maximum principle to $(x,t)\mapsto t^{-\frac12} \vthp_{x}(x,\big(4\pi t)^{-1}\big)$. Since 
\[
G_\lambda''(x) = 2\pi\lambda e^{-\pi\lambda x^2}(2\pi\lambda x^2-1),
\]
we have with the substitution $\lambda = (4\pi t)^{-1}$
\begin{align}\label{varth-rep}
\begin{split}
t^{-\frac12} \vthp_{x}(x,\lambda) &= t^{-\frac12} \sum_{n=1}^\infty G_\lambda''(n-x) = \sum_{n=1}^\infty k_{xx}(n-x,t).
\end{split}
\end{align}
In particular, $(x,t)\mapsto t^{-\frac12} \vthp_{x}\big(x,(4\pi t)^{-1}\big)$ is in the kernel of $L=\partial_{xx} -\partial_t$. Consider $\vthp_x$ on the set $\{(0,t):t>0\}$. Since $\vthp(0,\lambda)>0$, \eqref{vartheta-ineq1} implies that $\vthp(x,\lambda)\le \vthp(0,\lambda)$ for $x<0$. Hence for $x=0$
\begin{align}\label{bound0t}
\vthp_x\big(0,(4\pi t)^{-1}\big)\ge 0\qquad(t>0).
\end{align}
Consider next the vertical half line $\{(1/2,t):t>0\}$. We have
\begin{align}\label{thetau-rep}
\vthp_{x}\big(\tfrac{1}{2},\lambda\big) &= 2\pi \lambda \sum_{n=1}^\infty \left(2\pi \lambda\big(n-\tfrac{1}{2}\big)^2-1\right)\, e^{-\pi\lambda(n-\frac12)^2}.
\end{align}
From the product representation \eqref{th2-prod} we have 
\begin{equation*}
\theta_2\big(0,i\lambda^{-1}\big) = \prod_{n=1}^\infty\left(1-e^{-\frac{2\pi n}{\lambda}}\right)\prod_{n=1}^\infty\left(1-e^{-\frac{\pi(2n-1)}{\lambda}} \right)^2,
\end{equation*}
which implies in particular that $\lambda \mapsto \theta_2\big(0,i\lambda^{-1}\big) $ decreases, hence
\begin{equation}\label{der-th2}
\frac{\textrm{d}}{\dl}  \theta_2\big(0,i\lambda^{-1}\big) \leq 0
\end{equation}
for all positive $\lambda$. From the transformation formula for $\theta_2$ we obtain
\[
\theta_2\big(0,i \lambda^{-1}\big) = \lambda^{\frac12}\sum_{n=-\infty}^\infty e^{-\pi  \lambda (n+\frac12)^2}
\]
and differentiation with respect to $\lambda$ gives, together with \eqref{der-th2},
\begin{align*}
\begin{split}
0\ge \frac{\textrm{d}}{\dl}\theta_2\big(0,i \lambda^{-1}\big)&= \frac{1}{2\lambda^{\frac12}}\sum_{n=-\infty}^\infty e^{-\pi  \lambda (n+\frac12)^2} - \pi \lambda^{\frac12} \sum_{n=-\infty}^\infty \big(n+\tfrac{1}{2}\big)^2e^{-\pi  \lambda (n+\frac12)^2}\\
&=\frac{1}{2\lambda^{\frac12}} \sum_{n=-\infty}^\infty  \left(1-2\pi \lambda\big(n+\tfrac{1}{2}\big)^2\right)  e^{-\pi\lambda (n+\frac12)^2},
\end{split}
\end{align*}
hence from \eqref{thetau-rep} we obtain that
\begin{align}\label{bound12t}
\vthp_{x}\big(\tfrac12,(4\pi t)^{-1}\big)\ge 0\qquad(t>0).
\end{align}
Finally, we note that $t^{-\frac12}\vthp_x(x,\lambda)$ has a continuous extension to $(x,\lambda)$ for $x\in\RR\backslash\{1,2,....\}$ according to \eqref{varth-rep}. Hence
\begin{align}\label{boundu0}
\lim_{t\to0+} t^{-\frac12} \vthp_x\big(x,(4\pi t)^{-1}\big) = 0\qquad(0\le x\le 1/2).
\end{align}
From \eqref{varth-rep}, \eqref{bound0t}, \eqref{bound12t}, and \eqref{boundu0} it follows that the assumptions of Lemma \ref{maxprinc} are satisfied for $-\vthp_x$, hence \eqref{vthp-increasing} follows. As mentioned above, this implies \eqref{vartheta3}.
\end{proof}

\section{Proof of Theorem \ref{truncated-ba-theorem} }%%%%%%%%%%%%%%%%%%%%%%%%%%%%%%%%%%%%%%%%%%%%%%%%%%%%%%%%%%%%%%%%%%%%%%%%%%%

Let $K$ be an entire function of exponential type $\pi$ such that $\gplus- K$ is in $L^1(\RR)$. The partial sums of the Fourier expansion
\[
\text{sgn}(\sin\pi x) = \lim_{N\to\infty} \frac{i}{\pi } \sum_{|n|\le N} \frac{1}{\big(n+\frac12\big)}\,e^{-2\pi i(n+\frac12)x}
\]
are uniformly bounded, hence, letting
\begin{align*}
\varphi_\lambda(x) := \gplus(x)-K(x),
\end{align*}
dominated convergence and the Paley-Wiener theorem give
\begin{align}\label{FT-trick}
\begin{split}
\int_{-\infty}^\infty \text{sgn}(\sin\pi x) (\gplus(x)-K(x))\,\dx &= \lim_{N\to\infty} \frac{i}{\pi } \sum_{|n|\le N} \frac{\ft{\varphi}_\lambda\big(n+\frac12\big) }{n+\frac12}\\
&= \frac{i}{\pi } \sum_{n=-\infty}^\infty \frac{\ft{\gplus}\big(n+\frac12\big) }{n+\frac12}.
\end{split}
\end{align}
Taking absolute values in the integral, and using \eqref{th1} and \eqref{FT-Gaussian-3}, we arrive at the lower bound \eqref{ba-bound}. As mentioned in the introduction, $K_{\lambda}^{+}$ is entire and has exponential type $\pi$, so it remains to show that the lower bound \eqref{ba-bound} is attained for $K=K_{\lambda}^{+}$ (which follows once \eqref{extremal-eq} is established) and that the inequality is strict for any other $K$ of exponential type $\pi$. 

\medskip

Fix $\lambda>0$. In order to show \eqref{extremal-eq} we define the remainder $z\mapsto I(z)$ by
\begin{align*}
I(z) = \frac{\pi}{\sin\pi z}\big(K_{\lambda}^{+}(z) - \gplus(z)\big),
\end{align*}
and we aim to prove that for all real $x\neq0$ we have $I(x)\le 0$. This is shown in Lemma \ref{non-pos-I-1} for negative $x$ and in Lemma \ref{non-pos-I-2} for positive $x$.

\begin{lemma}\label{non-pos-I-1} For all $x<0$ we have
\begin{align*}
I(x)\le 0.
\end{align*}
\end{lemma}

\begin{proof} Let $x<0$. Define $I_N$ by
\begin{align*}
I_N(x) = \sum_{n=1}^N (-1)^n G_\lambda(n)\left(\frac{1}{x-n} - \frac{1}{x}\right),
\end{align*}
and note that $I_N\to I$ uniformly on compact subsets of $\CC / \Z^{+}$. Applying \eqref{intrep-1} for non-negative $n$ we obtain 

\begin{equation*}
I_N(x) = 2\pi \lambda^{\frac32} \int_{-\infty}^\infty G_\lambda(x-t)\int_{-\infty}^0 e^{-2\pi\lambda tu} \big\{\thpn(u,\lambda) - G_\lambda(u)\thpn(0,\lambda)\big\}\,\du\,\dt.
\end{equation*}
By \eqref{thp-growth1}, we can find $c(\lambda)>0$ such that
\begin{align*}
|I_N(x)| & \leq 2\pi \lambda^{\frac32} \int_{-\infty}^\infty G_\lambda(x-t)\int_{-\infty}^0 e^{-2\pi\lambda tu} \, \big|\thpn(u,\lambda) - G_\lambda(u)\thpn(0,\lambda)\big|\,\du\,\dt \\
&\le c(\lambda) \int_{-\infty}^\infty G_\lambda(x-t)\int_{-\infty}^0 e^{-2\pi\lambda tu} \,G_\lambda(u)\, \du\,\dt.
\end{align*}
We split now the outer integral at $t=0$. For $t\le 0$ we estimate $|e^{-2\pi \lambda tu}|\le 1$ and note that the resulting double integral is finite. For $t\ge 0$ we use first that
\[
G_\lambda(x-t) e^{-2\pi\lambda  ut} G_\lambda(u) = G_\lambda(x) e^{2\pi\lambda xt}  G_\lambda(t+u),
\]
extend the inner integral over $u$ to $\RR$, and use finally that $e^{2\pi \lambda xt}$ is integrable since $xt<0$. Dominated convergence then implies that
\begin{align*}
I(x) = 2\pi \lambda^{\frac32} \int_{-\infty}^\infty \int_{-\infty}^0 e^{-2\pi\lambda tu} G_\lambda(x-t)\{\thp(u,\lambda)-G_\lambda(u)\thp(0,\lambda)\}\,\du \,\dt,
\end{align*}
and \eqref{thp-ineq-eq} implies that the right-hand side is non-positive.
\end{proof}

The calculation for positive $x$ is slightly more involved. We prove first an integral representation for $I(z)$ valid when $\Re z >0 $.

\begin{lemma} \label{IntRep-1}
For $\Re z>0$ we have
\begin{align}\label{intrep1to6}
I(z) = 2\pi\lambda^{\frac32}\sum_{i=1}^6 L_i(z)\,,
\end{align}
where
\begin{align*}
\begin{split}
L_1(z)&= \int_{-\infty}^0 G_{\lambda}(z-t) \int_{-\infty}^0 e^{-2\pi \lambda tu} \,\thp(u,\lambda)\,\du\,\dt\,,\\
L_2(z)&=-\int_{0}^\infty G_{\lambda}(z-t) \int_{0}^\infty e^{-2\pi \lambda tu} \, \thp(u,\lambda)\,\du\,\dt\,,\\
L_3(z)&= \int_{0}^\infty G_{\lambda}(z-t) \int_{0}^\infty  e^{-2\pi \lambda tu} \, G_\lambda(u)\,\thp(0,\lambda) \,\du\,\dt\,,\\
L_4(z)&= \int_{-\infty}^0 G_{\lambda}(z-t) \int_{0}^\infty  e^{-2\pi \lambda tu}\, G_\lambda(u) \,\thp(0,\lambda) \,\du\,\dt\,,\\
L_5(z)&= \int_{-\infty}^0 G_{\lambda}(z-t) \int_{-\infty}^0 e^{-2\pi \lambda tu} \,\big\{\thp(-u,\lambda)-G_\lambda(u)\big\} \,\du\,\dt\,,\\
L_6(z)&=  \int_{-\infty}^0 G_{\lambda}(z-t) \int_{0}^\infty e^{-2\pi \lambda tu}\, \big\{\thp(-u,\lambda)-G_\lambda(u)\big\}\,\du\,\dt.
\end{split}
\end{align*}
\end{lemma}

\begin{proof} The function $I$ is the limit as $N\to\infty$ of
\begin{align*}
I_N(z) = \sum_{n=1}^N (-1)^n\frac{G_\lambda(n)-G_\lambda(z)}{z-n} +\sum_{n=1}^N (-1)^{n+1} \frac{G_\lambda(n)}{z} +\sum_{n=-N}^0 (-1)^{n+1} \frac{G_\lambda(z)}{z-n},
\end{align*}
and we denote these three sums by $I_N = I_{1,N} + I_{2,N} + I_{3,N}$. Equation \eqref{intrep2-eq} implies
\begin{align*}%\label{pr-1}
\begin{split}
I_{1,N}(z) &=  2\pi\lambda^{\frac32} \int_{-\infty}^0 G_\lambda(z-t) \int_{-\infty}^0e^{-2\pi\lambda tu}\, \thpn(u,\lambda)\,\du\,\dt \\
&\qquad - 2\pi\lambda^{\frac32}\int_0^\infty G_\lambda(z-t) \int_0^\infty e^{-2\pi\lambda tu}\, \thpn(u,\lambda)\,\du\,\dt\\
&:= I_{1,1,N}(z)-I_{1,2,N}(z).
\end{split}
\end{align*}
In $I_{1,1,N}$ we may apply dominated convergence using \eqref{thp-growth1} and $|e^{-2\pi\lambda tu} |\le 1$. Define
\begin{align*}
\Gamma_\lambda(u) &= -\int_u^\infty G_\lambda(t) \,\dt
\end{align*}
and
\begin{align*}
T_N^+(u,\lambda)&= \sum_{n=1}^N (-1)^{n} \,\Gamma_\lambda(n-u).
\end{align*}
From the fact that 
$$|\Gamma_\lambda(u) - \Gamma_\lambda(u+1)|\le \max_{u \leq y \leq u+1}\{G_\lambda(y)\} \,,$$
we obtain $|T_N^+(u,\lambda)|\le c_\lambda$ for all $u \in \R$ and $N\in\NN$. We note that $\frac{\textrm{d}}{\du} T_N^+(u,\lambda) = \theta_N^+(u,\lambda)$. An integration by parts gives for $t>0$
\[
\int_0^\infty e^{-2\pi\lambda tu} \,\theta_N^+(u,\lambda)\,\du = -T_N^+(0,\lambda) +2\pi\lambda t\int_0^\infty e^{-2\pi\lambda tu}\, T_N^+(u,\lambda)\,\du,
\]
and an application of Lebesgue dominated convergence shows that the limit as $N\to \infty$ in $I_{1,2,N}$ may be evaluated under the integral sign.

Hence $I_{1,N}$ converges to
\begin{align}\label{in1}
\begin{split}
I_{1}(z) &=  2\pi\lambda^{\frac32} \int_{-\infty}^0 G_\lambda(z-t) \int_{-\infty}^0e^{-2\pi\lambda tu}\, \thp(u,\lambda)\,\du\,\dt \\
&\qquad - 2\pi\lambda^{\frac32}\int_0^\infty G_\lambda(z-t) \int_0^\infty e^{-2\pi\lambda tu}\, \thp(u,\lambda)\,\du \,\dt\\
&= 2\pi\lambda^{\frac32} \big\{ L_1(z) + L_2(z)\big\}.
\end{split}
\end{align}

\medskip

Since $\Re{z} > 0$, equation \eqref{intrep-2} with $w=0$ implies
\begin{align*}
I_{2,N}(z) = 2\pi\lambda^{\frac32} \int_{-\infty}^\infty G_\lambda(z-t) \int_0^\infty e^{-2\pi\lambda tu} \,G_\lambda(u)\, \thpn(0,\lambda) \,\du \,\dt\,,
\end{align*}
which converges to
\begin{align}\label{in2}
\begin{split}
I_2(z) &= 2\pi\lambda^{\frac32} \int_{-\infty}^\infty G_\lambda(z-t) \int_0^\infty e^{-2\pi\lambda tu}  \, G_\lambda(u) \,\thp(0,\lambda)\,\du \,\dt\\
&= 2\pi\lambda^{\frac32} \big\{ L_3(z) + L_4(z)\big\}.
\end{split}
\end{align}

\medskip

To estimate $I_{3,N}$ we use expression \eqref{gauss3} and the identity
\[
\sum_{n=-N}^0 (-1)^{n+1} G_\lambda(u-n) = \thpn(-u,\lambda) - G_\lambda(u).
\]
We have
\begin{align}\label{pr-2}
\begin{split}
I_{3,N}(z) &= 2\pi\lambda^{\frac32} \int_{-\infty}^0 G_\lambda(z-t) \int_{-\infty}^0 e^{-2\pi\lambda tu} \big\{\thpn(-u,\lambda) - G_\lambda(u)\big\}\,\du\,\dt\\
&\qquad +2\pi\lambda^{\frac32} \int_{-\infty}^0 G_\lambda(z-t) \int_0^\infty e^{-2\pi\lambda tu} \big\{\thpn(-u,\lambda) - G_\lambda(u)\big\}\,\du\,\dt.
\end{split}
\end{align}
With an analogous argument as the one used for $I_{1,N}$ above we may verify that dominated convergence can be applied to both terms on the right side of \eqref{pr-2}, giving, as $N \to \infty$, that $I_{3,N}$ converges to $I_3$ with
\begin{align}\label{in3}
I_3(z) = 2\pi\lambda^{\frac32} \big\{ L_5(z) + L_6(z) \big\},
\end{align}
thus \eqref{in1}, \eqref{in2} and \eqref{in3} imply \eqref{intrep1to6}.
\end{proof}

\begin{lemma}\label{non-pos-I-2} For all $x>0$ we have $I(x)\le 0$.
\end{lemma}

\begin{proof} We define $W_1 = L_1+L_5$, $W_2 = L_2+L_3$, and $W_3= L_4+L_6$. Lemma \ref{IntRep-1} implies for $\Re z>0$ that
\begin{align}\label{sumofW}
I(z) = 2\pi\lambda^{\frac32}\sum_{i=1}^3 W_i(z),
\end{align}
and we note that
\begin{align*}
\begin{split}
W_1(z)&= \int_{-\infty}^0 G_{\lambda}(z-t) \int_{-\infty}^0 e^{-2\pi \lambda tu} \big\{-\lambda^{-\frac12} \theta_1\big(u,i\lambda^{-1}\big)\big\} \,\du\,\dt,\\
W_2(z)&=\int_{0}^\infty G_{\lambda}(z-t) \int_{0}^\infty e^{-2\pi \lambda tu}  \big\{ G_\lambda(u)\,\thp(0,\lambda)- \thp(u,\lambda)\big\}\,\du\,\dt,\\
W_3(z)&= \int_{-\infty}^0 G_{\lambda}(z-t) \int_{0}^\infty  e^{-2\pi \lambda tu} \big\{G_\lambda(u) \,\thp(0,\lambda) +  \thp(-u,\lambda)-G_\lambda(u)\big\}\,\du\,\dt. \\
\end{split}
\end{align*}

\medskip

We shall show that $W_i(x)\le 0$ for $i=1,2,3$ and $x>0$. For $\lambda >0$, the function $u \mapsto \theta_1\big(u,i\lambda^{-1}\big)$ is an even, real-valued function with period $2$, satisfying
\begin{equation*}
\theta_1\big(u+1,i\lambda^{-1}\big) = - \theta_1\big(u,i\lambda^{-1}\big).
\end{equation*}
From the product representation \eqref{th1-prod-1} we know that $u \mapsto \theta_1\big(u,i\lambda^{-1}\big)$ has only simple zeros at the points $\Z +\hh$, and therefore $ \theta_1\big(u,i\lambda^{-1}\big) >0$ for all real values of $u$ in the open interval $-\hh < u < \hh$. This is sufficient to establish for $t<0$
\begin{align*}%\label{W1}
\begin{split}
 \int_{-\infty}^0 e^{-2\pi \lambda tu} &\, \big\{-\lambda^{-\frac12} \theta_1\big(u,i\lambda^{-1}\big)\big\} \,\du \\
=& \sum_{k=0}^{\infty} e^{4\pi\lambda t k} \int_{-2}^0 e^{-2\pi \lambda tu} \, \big\{-\lambda^{-\frac12} \theta_1\big(u,i\lambda^{-1}\big)\big\} \,\du \\
=&\frac{1}{\big(1- e^{4\pi\lambda t}\big)} \int_{-2}^0 e^{-2\pi \lambda tu} \, \big\{-\lambda^{-\frac12} \theta_1\big(u,i\lambda^{-1}\big)\big\} \,\du \leq 0,
\end{split}
\end{align*}
which implies 
\begin{equation*}
W_1(x) \leq 0 \qquad(x>0).
\end{equation*}

\medskip

For $u\geq 0$, using  \eqref{thp-to-th1} and \eqref{thp-ineq-eq} we obtain
\begin{align*}
\thp(0,\lambda)&G_\lambda(u) -\thp(u,\lambda)\\
&=  \lambda^{-\frac12}\theta_1\big(u,i\lambda^{-1}\big)+ \thp(0,\lambda)G_\lambda(u) +\thp(-u,\lambda) - G_\lambda(u)\\
&\le  \lambda^{-\frac12}\theta_1\big(u,i\lambda^{-1}\big) + \big(2 \thp(0,\lambda)-1\big)G_\lambda(u) \\
&= \lambda^{-\frac12}\theta_1\big(u,i\lambda^{-1}\big) -\lambda^{-\frac12}\theta_1\big(0,i\lambda^{-1}\big) G_\lambda(u),
\end{align*}
and \eqref{needed-for-l2l3} implies
\begin{align*}
W_2(x) \leq 0
\end{align*}
for $x>0$.

\medskip

To prove $W_3(x)\le 0$ we consider $u\ge 0$ and use \eqref{thp-to-th1} and \eqref{thp-ineq-eq} to get
\begin{align*}
\begin{split}
\thp(0,\lambda)G_\lambda(u) &+\thp(-u,\lambda) -G_\lambda(u)\\
& \le \big(2\thp(0,\lambda) -1\big)\,G_\lambda(u)=-\lambda^{-\frac12}\,\theta_1\big(0,i\lambda^{-1}\big) \,G_\lambda(u)\le 0,
\end{split}
\end{align*}
and this shows that 
\begin{align*}
W_3(x) \leq 0
\end{align*}
for $x>0$. 
\end{proof}

It remains to show that if equality holds in \eqref{ba-bound} for some $K$ of exponential type $\pi$, then $K=K_{\lambda}^{+}$. It follows from \eqref{ba-bound} and \eqref{FT-trick} that for such a function $K$ we must have
\begin{align*}
\int_{-\infty}^\infty \text{sgn}(\sin\pi x)\big\{\gplus(x) - K(x)\big\}\,\dx = \int_{-\infty}^\infty \big|\gplus(x) - K(x)\big|\,\dx.
\end{align*}
The function $x\mapsto \gplus(x) - K(x)$ is continuous for all $x\neq 0$, and hence
\begin{align*}
\text{sgn}(\sin\pi x)\big\{\gplus(x) - K(x)\big\} = \big|\gplus(x) - K(x)\big|
\end{align*}
for $x\neq 0$. It follows that for all $n\in\ZZ\backslash\{0\}$
\begin{align*}
K(n) = \gplus(n) = K_{\lambda}^{+}(n),
\end{align*}
and therefore by a standard interpolation theorem for functions of exponential type $\pi$  \cite[Vol. II, p. 275]{Z}
\begin{align*}
K(z) - K_{\lambda}^{+}(z) = \big(K(0)-K_{\lambda}^{+}(0)\big)\frac{\sin\pi z}{\pi z},
\end{align*}
for all complex $z$. Since $K-K_{\lambda}^{+}$ is integrable, we obtain $K(0) = K_{\lambda}^{+}(0)$ and therefore $K=K_{\lambda}^{+}$.

\section{Proof of Theorem \ref{thm-one-sided}}%%%%%%%%%%%%%%%%%%%%%%%%%%%%%%%%%%%%%%%%%%%%%%%%%%%%%%%%%%%%%%%%%%%%%%%%%%%

\textsl{Proof of Theorem \ref{thm-one-sided} (i).} Let $L: \C \to \C$ be an entire function of exponential type at most $2\pi$, that is real and integrable on $\R$, and satisfies $\gplus(x) \geq L(x)$, for all $x \in \R$. From Poisson summation formula and the Paley-Wiener theorem we have
\begin{align}\label{pf-thm2-1}
\begin{split}
\int_{-\infty}^{\infty} L(x)\,\dx &= \widehat{L}(0) = \sum_{n=-\infty}^{\infty} L(n) \leq   \sum_{n=1}^{\infty} \gplus(n)  =  \frac{\theta_3(0, i \lambda)}{2} - \frac12\,,
\end{split}
\end{align}
and hence
\begin{equation}\label{pf-thm2-2}
\int_{-\infty}^{\infty} \big\{\gplus(x) - L(x) \big\}\,\dx \geq  -\frac{\theta_3(0, i \lambda)}{2} + \frac12\ + \frac{1}{2\sqrt{\lambda}}.
\end{equation}
Observe that $\lplus$ satisfies the equalities in \eqref{pf-thm2-1} and \eqref{pf-thm2-2} since it interpolates $\gplus$ at $\Z\backslash\{0\}$ and is equal to $0$ at $x=0$.

\medskip

We now move to the proof of the inequality $\gplus(x) \geq \lplus(x)$ for all $x\in \R$. We start by defining $z\mapsto R(z)$ by
\begin{equation*}
R(z) = \lplus(z) - \gplus(z).
\end{equation*}

\begin{lemma} The inequality
\begin{align}\label{ineq-xneg}
R(x) \le 0
\end{align}
holds for all $x<0$.
\end{lemma}

\begin{proof} We define
\begin{align}\label{defrn}
R_N(z) = \frac{\sin^2\pi z}{\pi^2}\sum_{n=1}^N \left\{ \frac{G_\lambda(n)}{(z-n)^2} + \frac{G_\lambda'(n)}{z-n}-\frac{G_\lambda'(n)}{z}\right\}
\end{align}
and note that $R_N\to R$ uniformly on compact sets in $\Re z <0$. Recall that
\begin{align*}
\vthp(u,\lambda) = \sum_{n=1}^\infty G_\lambda'(u-n)
\end{align*}
and $\vthp_N$ are the partial sums of $\vthp$. We differentiate \eqref{intrep-1} with respect to $w$. The resulting representation with $w=n$ is used to replace the first two terms in each summand of \eqref{defrn}. The third term in each summand is expanded using \eqref{intrep-1} with $w=0$. In this way we obtain  
\begin{align*}
 \frac{R_N (z)}{2\pi\lambda^{\frac32}} &= - \frac{\sin^2\pi z}{\pi^2}\,\sum_{n=1}^N \int_{-\infty}^\infty \int_{-\infty}^0 e^{-2\pi\lambda tu} G_\lambda(z-t)\\
 & \hspace*{3cm} \big\{G_\lambda'(n-u) - G_\lambda'(n)G_\lambda(u)\big\} \,\du \,\dt\\
&=  \frac{\sin^2\pi z}{\pi^2}\,\int_{-\infty}^\infty \int_{-\infty}^0 e^{-2\pi\lambda tu} G_\lambda(z-t) \big\{\vthp_N(u,\lambda) - \vthp_N(0,\lambda) G_\lambda(u)\big\}\,\du \,\dt.
\end{align*}

It can be checked similarly to the proof of Lemma \ref{non-pos-I-1} that the assumptions of the dominated convergence theorem are satisfied. It follows in  $\Re z<0$ that
\begin{align}\label{Jm-rep}
\begin{split}
R(z)&= 2\pi\lambda^{\frac32}\, \frac{\sin^2\pi z}{\pi^2}\int_{-\infty}^\infty \int_{-\infty}^0 \,e^{-2\pi\lambda tu} G_\lambda(z-t)\big \{\vthp(u,\lambda) - \vthp(0,\lambda) G_\lambda(u) \big\} \,\du\dt,
\end{split}
\end{align}
and an application of \eqref{vartheta-ineq1} gives \eqref{ineq-xneg}.
\end{proof}

The proof of $R(x)\le 0$ for $x>0$ is separated in two lemmata. We first give an integral representation for $R$ in Lemma \ref{intrep-R} valid in the region $0<\Re z$, then prove non-negativity in Lemma \ref{ext-min-nonpositivity}.

\begin{lemma} \label{intrep-R}
For $0<\Re z$ we have
\begin{align}\label{sum-of-Is}
R(z)  = 2\pi\lambda^{\frac32} \, \frac{\sin^2\pi z}{\pi^2}\,\sum_{k=1}^6 S_k(z)\,,
\end{align}
where
\begin{align}\label{def-of-theIs}
\begin{split}
S_1(z)&= \int_{-\infty}^0 \int_{-\infty}^0 e^{-2\pi\lambda tu} \, G_\lambda(z-t) \,\vthp(u,\lambda) \,\du \,\dt\,, \\
S_2(z)&= -\int_0^\infty \int_0^\infty e^{-2\pi\lambda tu} \,G_\lambda(z-t)\, \vthp(u,\lambda) \,\du \,\dt\,,\\
S_3(z) &= \int_{-\infty}^0 \int_0^\infty e^{-2\pi\lambda t u} \,G_\lambda(z-t) \,G_\lambda(u)\,\vthp(0,\lambda)\,\du \,\dt\,,\\ 
S_4(z)&=  \int_{0}^\infty \int_0^\infty e^{-2\pi\lambda t u}\, G_\lambda(z-t) \,G_\lambda(u)\,\vthp(0,\lambda) \,\du \,\dt\,,\\ 
S_5(z)&= \int_{-\infty}^0 \int_{-\infty}^0 e^{-2\pi\lambda t u} \,G_\lambda(z-t) \,\big\{G_\lambda'(u) - \vthp(-u,\lambda)\big\} \,\du \,\dt\,,\\ 
S_6(z) &=   \int_{-\infty}^0 \int_0^\infty  e^{-2\pi\lambda t u} \,G_\lambda(z-t) \,\big\{ G_\lambda'(u) - \vthp(-u,\lambda)\big\} \,\du \,\dt.
\end{split}
\end{align}
\end{lemma}

\begin{proof}
Let $ 0 < \Re z$ and define
\begin{align*}
\begin{split}
R_N(z) &= \, \frac{\sin^2\pi z}{\pi^2}\,\left(\sum_{n=1}^N \left\{ \frac{G_\lambda(n) -G_\lambda(z)}{(z-n)^2} +\frac{G_\lambda'(n)}{z-n} -\frac{G_\lambda'(n)}{z}\right\}- \sum_{n=-N}^0 \frac{G_\lambda(z)}{(z-n)^2}\right).
\end{split}
\end{align*}
We note that $R_N\to R$ uniformly in compact sets in $0<\Re z$ as $N\to\infty$. We differentiate \eqref{intrep2-eq} and \eqref{gauss3} with respect to $w$. Together with \eqref{intrep-2} we obtain 
\begin{align}\label{expansion-sum}
\begin{split} 
R_N(z) = 2\pi\lambda^{\frac32} &\, \frac{\sin^2\pi z}{\pi^2}\,\left\{\int_{-\infty}^0 \int_{-\infty}^0 \,e^{-2\pi\lambda tu} \,G_\lambda(z-t) \,\vthp_N(u,\lambda) \,\du \,\dt\right.\\ 
& - \int_0^\infty \int_0^\infty e^{-2\pi\lambda tu} \,G_\lambda(z-t) \,\vthp_N(u,\lambda) \,\du \,\dt\\
 &+ \int_{-\infty}^\infty \int_{0}^\infty e^{-2\pi\lambda t u} \,G_\lambda(z-t) \,G_\lambda(u)\,\vthp_N(0,\lambda) \,\du \,\dt\\ 
& -  \int_{-\infty}^0 \int_{-\infty}^0 e^{-2\pi\lambda t u} \,G_\lambda(z-t) \sum_{n=-N}^0 G_\lambda'(n-u) \,\du \,\dt\\ 
& \left.- \int_{-\infty}^0 \int_0^\infty  e^{-2\pi\lambda t u} \,G_\lambda(z-t)\sum_{n=-N}^0 G_\lambda'(n-u) \,\du \,\dt \right\}.
\end{split}
\end{align}
We note that
\[
\sum_{n=-N}^0 G_\lambda'(n-u) = \sum_{n=0}^N G_\lambda'(-n-u)= \vthp_N(-u,\lambda) -G_\lambda'(u)\,, 
\]
which gives rise to a representation
\begin{align*}
R_N(z)= 2\pi\lambda^{\frac32}\,\frac{\sin^2\pi z}{\pi^2}\, \sum_{k=1}^6 S_{k,N}(z),
\end{align*}
where the functions $S_{k,N}$ are defined by replacing the series in \eqref{def-of-theIs} by their respective partial sums. Note in particular that the third integral in \eqref{expansion-sum} equals $S_{3,N}+S_{4,N}$.

\medskip

It remains to justify the change of integration and limit as $N\to\infty$. For $S_{1,N}$, $S_{3,N}$, and $S_{4,N}$, this is straightforward; we omit the calculations. To apply dominated convergence in $S_{6,N}$ we note that for $t<0$ and $u>0$ by \eqref{II-theta-N-ub}
\begin{align*}
\big| e^{-2\pi\lambda t u}& G_\lambda(z-t)\big\{ \vthp_N(-u,\lambda) - G_\lambda'(u) \big\}\big|\\
&\le c \,(|u|+1|) \, e^{-2\pi\lambda tu} \, e^{-\pi\lambda(\Re z-t)^2} \, G_{\lambda}(u)\\
&= c \, (|u|+1) \, e^{-\pi\lambda \Re z^2} \, e^{2\pi\lambda t\Re z } \,G_\lambda(t+u)\,,
\end{align*}
with a constant $c= c(\lambda, z)>0$, and since $\Re z>0$, the latter expression is in $L^1((-\infty,0]\times [0,\infty))$, which finishes the proof for $S_{6,N}$. 

\medskip

To deal with $S_{2,N}$ we use integration by parts
\begin{align}\label{int-parts-I}
\begin{split}
&S_{2,N}(z) =  - \int_0^\infty G_\lambda(z-t) \,\sum_{n=1}^N  \,\int_0^\infty  e^{-2\pi\lambda tu} \, G_{\lambda}'(u-n)\,\du \,\dt\\
& = - \int_0^\infty G_\lambda(z-t) \,\sum_{n=1}^N  \left\{-G_{\lambda}(n)  + 2\pi \lambda t \int_0^{\infty}  \,e^{-2\pi\lambda t  u}\, G_{\lambda}(u-n)\,\du     \right\}\dt\,,
\end{split}
\end{align}
and since 
\begin{equation*}
\sum_{n=1}^N G_{\lambda}(u-n) \leq C
\end{equation*}
for all $u \in \R$ and all $N$, we can pass to the limit as $N\to\infty$ in \eqref{int-parts-I} and use integration by parts again to get
\begin{align*}
S_{2}(z) & = -\int_0^\infty G_\lambda(z-t) \int_0^{\infty} e^{-2\pi\lambda t  u}\, \vthp(u,\lambda)\,\du \, \dt.
\end{align*}
With a similar argument, using integration by parts twice we show that 
\begin{align*}
S_{5}(z) & = \int_{-\infty}^0 \int_{-\infty}^0 e^{-2\pi\lambda t u} G_\lambda(z-t) \big\{G_\lambda'(u) - \vthp(-u,\lambda)\big\} \,\du \,\dt\,,
\end{align*}
which finishes the proof.
\end{proof}

\begin{lemma}\label{ext-min-nonpositivity}
 Let $x>0$. Then
\begin{align*}
R(x)\le 0.
\end{align*}
\end{lemma}

\begin{proof} We combine the integrals in \eqref{def-of-theIs} by integration region. We note that
\begin{align}\label{repi2i4}
S_2(z) + S_4(z)&= \int_0^\infty \int_0^\infty e^{-2\pi\lambda tu} G_\lambda(z-t) \big\{\vthp(0,\lambda) G_\lambda(u)-\vthp(u,\lambda) \big\} \,\du \,\dt.
\end{align}
We split the integral over $u$ at $u=1/2$ and replace $\vthp(u,\lambda)$ for $u\ge 1/2$ using \eqref{vthp-to-th3} to arrive at
\begin{align}\label{I2I4}
\begin{split}
S_2(z) & + S_4(z)= \int_0^\infty \int_0^{\frac12} e^{-2\pi\lambda tu} \,G_\lambda(z-t)\, \big\{\vthp(0,\lambda) G_\lambda(u)-\vthp(u,\lambda) \big\} \,\du \,\dt\\
& +\int_0^\infty \int_{\frac12}^\infty e^{-2\pi\lambda tu} \,G_\lambda(z-t)\,\big\{\vthp(0,\lambda) G_\lambda(u) - \vthp(-u,\lambda) +G_\lambda'(u)\big\} \,\du \,\dt\\
 & - 2 \lambda^{-\frac12}\int_0^\infty \,G_\lambda(z-t) \int_0^{\frac12} \frac{\sinh(2\pi\lambda tu)}{1-e^{2\pi\lambda t}} \,\theta_3'\big(u,i\lambda^{-1}\big) \,\du \,\dt,
\end{split}
\end{align}
where we have used in \eqref{I2I4} the fact that, for positive $t$, 
\begin{align}\label{laplace-trick-1}
\begin{split}
\int_{\frac12}^\infty e^{-2\pi\lambda tu} \,\theta_3'\big(u,i\lambda^{-1}\big) \,\du &=
\sum_{n=1}^\infty e^{-2\pi\lambda tn} \int_{-\frac12}^{\frac12} e^{-2\pi\lambda tu} \, \theta_3'\big(u,i\lambda^{-1}\big) \,\du\\
&= 2 \int_0^{\frac12} \frac{\sinh(2\pi\lambda tu)}{1-e^{2\pi\lambda t}} \,\theta_3'\big(u,i\lambda^{-1}\big) \,\du,\\
\end{split}
\end{align}
since $u \mapsto \theta_3'\big(u,i\lambda^{-1}\big) $ is odd and 1-periodic. Returning to \eqref{I2I4}, an application of \eqref{vartheta3} in the first integral, \eqref{vartheta2} in the second integral, and Lemma \ref{theta2} in the third integral implies that for $x>0$ we have
\[
S_2(x) + S_4(x)\le 0.
\]

\medskip

We show next that $S_1+S_3+S_5+S_6\le 0$. We estimate first $S_1+S_5$ and $S_3+S_6$ separately. We have with an application of \eqref{vthp-to-th3}
\begin{align}
S_1(z) + S_5(z)&= \int_{-\infty}^0\int_{-\infty}^0 e^{-2\pi\lambda tu} G_\lambda(z-t) \big\{\vthp(u,\lambda) + G_\lambda'(u) -\vthp(-u,\lambda)\big\}\,\du \,\dt \nonumber \\
&= \lambda^{-\frac12} \int_{-\infty}^0\int_{-\infty}^0 e^{-2\pi\lambda tu} \,G_\lambda(z-t)\,\theta_3'\big(u,i\lambda^{-1}\big) \,\du \,\dt \nonumber \\
\begin{split} \label{repi1i5}
&= \lambda^{-\frac12} \int_{-\infty}^0\int_{-\frac12}^0 e^{-2\pi\lambda tu} \,G_\lambda(z-t)\,\theta_3'\big(u,i\lambda^{-1}\big) \,\du \,\dt \\
& \ \ \ + 2\lambda^{-\frac12} \int_{-\infty}^0 G_{\lambda}(z-t) \int_0^{\frac12} \frac{\sinh(2\pi\lambda tu)}{1-e^{-2\pi\lambda t}} \, \theta_3'\big(u,i\lambda^{-1}\big)\,\du\,\dt,
\end{split}
\end{align}
where we have used the identity, for negative $t$,
\begin{align}\label{laplace-trick-2}
\begin{split}
\int_{-\infty}^{-\frac12} e^{-2\pi\lambda tu} \,\theta_3'\big(u,i\lambda^{-1}\big) \,\du &=
\sum_{n=1}^\infty e^{2\pi\lambda tn} \int_{-\frac12}^{\frac12} e^{-2\pi\lambda tu} \, \theta_3'\big(u,i\lambda^{-1}\big) \,\du\\
&= 2 \int_0^{\frac12} \frac{\sinh(2\pi\lambda tu)}{1-e^{-2\pi\lambda t}} \,\theta_3'\big(u,i\lambda^{-1}\big) \,\du.\\
\end{split}
\end{align}
Since $(1-e^{-2\pi\lambda t})^{-1} \sinh(2\pi\lambda tu)\ge 0$ for $t<0$ and $u>0$, an application of  Lemma \ref{theta2} in the last integral of \eqref{repi1i5} implies for $x>0$ that
\begin{align}\label{S15-estimate}
\begin{split}
S_1(x) + S_5(x)&\le \lambda^{-\frac12} \int_{-\infty}^0\int_{-\frac12}^0 e^{-2\pi\lambda tu} \,G_\lambda(x-t)\, \theta_3'\big(u,i\lambda^{-1}\big) \,\du \,\dt\\
&= -\lambda^{-\frac12} \int_{-\infty}^0 \int_0^{\frac12} e^{2\pi\lambda tu} \,G_\lambda(x-t)\, \theta_3'\big(u,i\lambda^{-1}\big) \,\du \,\dt.
\end{split}
\end{align}
The expression on the right-hand side of \eqref{S15-estimate} is non-negative, but it turns out to cancel with $S_3+S_6$. We note that
\begin{align}\label{repi3i6}
\begin{split}
 S_3(z) + S_6(z)= \int_{-\infty}^0 \int_0^\infty  e^{-2\pi\lambda tu} \, G_\lambda(z-t)\, \big\{& \vthp(0,\lambda)  G_\lambda(u) \\
&+ G_\lambda'(u) -\vthp(-u,\lambda)\big\}\,\du\,\dt.
\end{split}
\end{align}
We obtain from \eqref{vartheta2} with $-u<0$ the inequality
\[
\vthp(0,\lambda) G_\lambda(u) + G_\lambda'(u) -\vthp(-u,\lambda) \le \frac{G_\lambda'(-u)}{2} +G_\lambda'(u) = \frac{G_\lambda'(u)}{2}\le 0
\]
and applying this for $x>0$ and $u\ge 1/2$ gives the upper bound
\begin{align}\label{S36-estimate}
S_3(x) + S_6(x)\le \int_{-\infty}^0 \int_0^{\frac12} e^{-2\pi\lambda tu} \,G_\lambda(x-t)\,\big\{&\vthp(0,\lambda)  G_\lambda(u) \\
& + G_\lambda'(u) -\vthp(-u,\lambda)\big\}\,\du\,\dt.\nonumber
\end{align}
We combine now \eqref{S15-estimate} and \eqref{S36-estimate}. For $t<0$ and $u>0$ we have $e^{2\pi\lambda tu} \le e^{-2\pi\lambda tu}$. An application of Lemma \ref{theta2}, \eqref{vthp-to-th3} and \eqref{vartheta3} gives
\begin{align*}
  \big\{\vthp(0,\lambda) &G_\lambda(u) + G_\lambda'(u) -\vthp(-u,\lambda)\big\}e^{-2\pi\lambda tu}- \lambda^{-\frac12}  \theta_3'\big(u,i\lambda^{-1}\big)e^{2\pi\lambda tu}\\
  &\le e^{-2\pi\lambda tu}\big\{ \vthp(0,\lambda) G_\lambda(u) + G_\lambda'(u) -\vthp(-u,\lambda)-\lambda^{-\frac12}  \theta_3'\big(u,i\lambda^{-1}\big)\big\}\\
 &= e^{-2\pi\lambda tu}\big\{\vthp(0,\lambda) G_\lambda(u) -\vthp(u,\lambda)\big\}\le 0,
\end{align*}
and hence
\[
S_1(x)+S_3(x)+S_5(x)+S_6(x)\le 0
\]
for $x>0$.
\end{proof}

The uniqueness part follows from classical arguments in this theory. Suppose that $L: \C \to \C$ is an entire function of exponential type at most $2\pi$, real and integrable on $\R$ such that $\gplus(x) \geq L(x)$, for all $x \in \R$. If equality happens on \eqref{pf-thm2-2} we must have
\begin{equation*}
L(n) = \gplus(n) = \lplus(n),
\end{equation*}
for all $n \in \Z\backslash\{0\}$, and $L(0) = 0 = \lplus(0)$. Since $L$ minorizes $\gplus$ this implies also
\begin{equation*}
L'(n) = \big(G_{\lambda}^{+}\big)'(n) = \big(L_{\lambda}^{+}\big)'(n),
\end{equation*}
for all $n \in \Z\backslash\{0\}$. Therefore the entire function
$$ z \mapsto \lplus(z) - L(z)$$
has exponential type at most $2\pi$, vanishes at each point of $\Z$ and its derivative vanishes at each point of $\Z/\{0\}$. An application of \cite[Lemma 4]{GV} shows that this function must be identically zero, and thus $L = \lplus$.

\medskip

\textsl{Proof of Theorem \ref{thm-one-sided} (ii).} The proof of the minimal integral and the uniqueness statement of (ii) follow by analogous arguments as for part (i) and are omitted. It remains to show that 
$$\mplus(x) \geq \gplus(x)$$
for all $x \in \R$. For this we define the difference function
$$ T(z)  = \mplus(z) - \gplus(z)$$
for $z\in\C$ and the desired inequality follows from the two results below.

\begin{lemma} The inequality
\[
T(x)\ge 0
\]
holds for all $x<0$.
\end{lemma}

\begin{proof} We note the identity
\begin{align}\label{rm-to-rp}
T(z) = R(z) +\frac{\sin^2\pi z}{\pi^2 z^2}.
\end{align}
Differentiation of \eqref{intrep-1} with respect to $w$ and setting $w=0$ gives for $\Re z<0$
\begin{align}\label{int1aw0}
\frac{1}{z^2} = 2\pi\lambda^{\frac32} \int_{-\infty}^\infty \int_{-\infty}^0 e^{-2\pi\lambda tu} \,G_\lambda(z-t) \,G_\lambda'(u) \,\du \,\dt.
\end{align}
Plugging \eqref{int1aw0} and \eqref{Jm-rep} into \eqref{rm-to-rp} gives for all $z$ with $\Re z <0$ the representation
\begin{align*}
\begin{split}
T(z) = 2\pi\lambda^{\frac32} \, \frac{\sin^2\pi z}{\pi^2}\,\int_{-\infty}^\infty \int_{-\infty}^0 &e^{-2\pi\lambda tu} G_\lambda(z-t) \\
& \big\{\vthp(u,\lambda) + G_\lambda'(u) -\vthp(0,\lambda) G_\lambda(u) \big\} \,\du \,\dt.
\end{split}
\end{align*} 
Inequality \eqref{vartheta2} implies
\[
\vthp(u,\lambda) -\vthp(0,\lambda) G_\lambda(u)+G_\lambda'(u) \ge -\frac12 G_\lambda'(u)+G_\lambda'(u)\ge 0,
\]
which proves the lemma.
\end{proof}

\begin{lemma} The inequality
\begin{align*}
T(x) \ge 0
\end{align*}
holds for all $x>0$.
\end{lemma}

\begin{proof} Differentiation of \eqref{intrep-2} with respect to $w$ and setting $w=0$ gives for $\Re z>0$ the representation
\begin{align*}
\frac{1}{z^2} =-2\pi\lambda^{\frac32} \int_{-\infty}^\infty \int_0^\infty e^{-2\pi\lambda tu}\, G_\lambda(z-t)  \,G_\lambda'(u) \,\du \,\dt.
\end{align*}
Identities \eqref{def-of-theIs}, \eqref{repi2i4}, \eqref{repi1i5} and \eqref{repi3i6} lead to 
\begin{equation*}
T(z) = 2\pi\lambda^{\frac32} \, \frac{\sin^2\pi z}{\pi^2}\, (V_1+V_2+V_3),
\end{equation*}
where
\begin{align*}
\begin{split}
V_1(z) &= \int_0^\infty \int_0^\infty e^{-2\pi\lambda tu} \,G_\lambda(z-t) \big\{\vthp(0,\lambda) G_\lambda(u) -\vthp(u,\lambda) -G_\lambda'(u) \big\}\,\du \,\dt,\\
V_2(z) &= \int_{-\infty}^0 \int_{-\infty}^0 e^{-2\pi\lambda tu} \,G_\lambda(z-t)  \,\big\{\lambda^{-\frac12} \theta_3'\big(u,i\lambda^{-1}\big)\big\} \,\du \,\dt,\\
V_3(z) &= \int_{-\infty}^0 \int_0^\infty e^{-2\pi\lambda tu} \,G_\lambda(z-t) \,\big\{\vthp(0,\lambda) G_\lambda(u) -\vthp(-u,\lambda)\big\} \,\du \,\dt.
\end{split}
\end{align*}
An application of \eqref{vthp-to-th3} and \eqref{vartheta-ineq1} for $-u<0$ gives
\begin{align}\label{th3-add}
\begin{split}
\vthp(0,\lambda) G_\lambda(u) -\big\{\vthp&(u,\lambda) +G_\lambda'(u)\big\} \\
&= \vthp(0,\lambda) G_\lambda(u) - \big\{\lambda^{-\frac12} \theta_3'\big(u,i\lambda^{-1}\big) +\vthp(-u,\lambda)\big\}\\
&\ge -\lambda^{-\frac12} \theta_3'\big(u,i\lambda^{-1}\big).
\end{split}
\end{align}
Plugging \eqref{th3-add} back into $V_1(x)$ and performing a calculation analogous to \eqref {laplace-trick-1} leads to $V_1(x)\ge 0$ for all $x >0$. In a similar way, using the rationale in  \eqref {laplace-trick-2}, we arrive at $V_2(x)\ge 0$ for $x>0$. Finally, an application of \eqref{thp-ineq-eq} with $-u<0$ implies that $V_3(x)\ge 0$ for $x>0$.
\end{proof}

\section{Asymptotic analysis}\label{Asy}%%%%%%%%%%%%%%%%%%%%%%%%%%%%%%%%%%%%%%%%%%%%%%%%%%%%%%%%%%%%%%%%%%%%%%%%%%%%%%%%%%

We are now interested in understanding the set of admissible non-negative Borel measures $\nu$ on $[0,\infty)$ against which we can integrate the minimal integral appearing in Theorem \ref{truncated-ba-theorem}. We define $H$ by
\begin{equation*}%\label{asymp-1}
H(\lambda) = \frac{1}{\pi \lambda} \int_0^1 \theta_1\left(0, i \lambda ^{-1}\big(1-y^2\big) \right) \dy.
\end{equation*}
and provide in this section a brief asymptotic analysis of this expression.

\begin{lemma}\label{asymp-lemma}
The function $H$ satisfies
\begin{align}\label{asymp-lemma-1}
\lim_{\lambda \to \infty} \lambda^{1/2} H(\lambda) =\frac{1}{2}
\end{align}
and 
\begin{align}\label{asymp-lemma-2}
\lim_{\lambda \to 0} H(\lambda) = \frac{1}{2}.
\end{align}
\end{lemma}
\begin{proof} 
By the transformation formula \eqref{th1-series} we see that
\begin{equation*}
 \left(\frac{\lambda}{1-y^2}\right)^{-1/2} \theta_1\left(0, i \lambda ^{-1}\big(1-y^2\big) \right) =  \sum_{n=-\infty}^\infty (-1)^n G_{\lambda(1-y^2)^{-1}}(n)
\end{equation*}
and, since the Gaussian is radially decreasing, this implies
\begin{equation*}%\label{asymp-2}
 \left( 1- 2 e^{-\pi \lambda (1-y^2)^{-1}}\right) \leq \left(\frac{\lambda}{1-y^2}\right)^{-1/2} \theta_1\left(0, i \lambda ^{-1}\big(1-y^2\big) \right) \leq 1
 \end{equation*}
for all $\lambda >0$ and $y \in (0,1)$. We arrive at 
\begin{align}\label{asymp-3}
\begin{split}
 \frac{1}{\pi} \int_0^1 \frac{1}{(1-y^2)^{1/2}} \left( 1- 2 e^{-\pi \lambda (1-y^2)^{-1}}\right)  \dy  \leq \lambda^{1/2} H(\lambda) \leq \frac{1}{\pi} \int_0^1 \frac{1}{(1-y^2)^{1/2}}\,\dy.
\end{split}
\end{align}
Using dominated convergence as $\lambda \to \infty$ and a direct evaluation, it follows that both integrals in \eqref{asymp-3} converge to the value $1/2$, which finishes the proof of \eqref{asymp-lemma-1}.

\medskip

The proof of \eqref{asymp-lemma-2} is slightly more involved. We define for each $t \in \R$ the function 
\begin{equation*}
H_t(\lambda)= \frac{1}{\pi \lambda} \int_0^1 e^{-\pi \lambda^{-1} t^2 (1-y^2)}\,\dy\,,
\end{equation*}
and note that for each $\lambda >0$ we have
\begin{equation*}
H(\lambda) = \sum_{n=-\infty}^{\infty} H_{n+\frac{1}{2}}(\lambda).
\end{equation*}
For each $t\neq 0$ we have (using $(1-y) \leq (1-y^2) \leq 2(1-y)$ for $0\le y\le 1$)
\begin{align}\label{asymp-dc}
\begin{split}
H_t(\lambda) \leq  \frac{1}{\pi \lambda} \int_0^1 e^{-\pi \lambda^{-1} t^2 (1-y)}\,\dy &=  \frac{1}{\pi \lambda} \int_0^1 e^{-\pi \lambda^{-1} t^2 w}\,\dw\\
& \leq \frac{1}{\pi \lambda} \int_0^{\infty} e^{-\pi \lambda^{-1} t^2 w}\,\dw = \frac{1}{\pi^2 t^2}\,,
\end{split}
\end{align}
and 
\begin{align}\label{asymp-lb}
H_t(\lambda) \geq  \frac{1}{\pi \lambda} \int_0^1 e^{-2\pi \lambda^{-1} t^2 (1-y)}\,\dy &= \frac{1}{2 \pi ^2 t^2} \left( 1 - e^{-2 \pi \lambda^{-1} t^2} \right).
\end{align}
We are interested in evaluating the limit of $H_t(\lambda)$ as $\lambda \to 0$. For $t\neq 0$, let us split the integral in two parts 
\begin{equation*}
H_t(\lambda)= \frac{1}{\pi \lambda}\left\{ \int_0^a e^{-\pi \lambda^{-1} t^2 (1-y^2)}\,\dy +  \int_a^1 e^{-\pi \lambda^{-1} t^2 (1-y^2)}\,\dy \right\},
\end{equation*}
where $a$ is to be chosen later. In the first integral we use the fact that $(1-y^2) \geq (1-a^2)$, while in the second integral we use $(1-y^2) \geq (1+a)(1-y)$ to obtain the upper bound
\begin{align}\label{asymp-4}
\begin{split}
H_t(\lambda) & \leq \frac{1}{\pi \lambda}\left\{ \int_0^a e^{-\pi \lambda^{-1} t^2 (1-a^2)}\,\dy  +  \int_a^1 e^{-\pi \lambda^{-1} t^2 (1+a)(1-y)}\,\dy \right\}\\
& = \frac{1}{\pi \lambda}\left\{  a e^{-\pi \lambda^{-1} t^2 (1-a^2)}  + \frac{1}{\pi \lambda^{-1} t^2 (1+a)} \left( 1- e^{-\pi \lambda^{-1} t^2 (1+a)(1-a)} \right)\right\}\\
& = \frac{ a e^{-\pi \lambda^{-1} t^2 (1-a^2)} }{\pi \lambda} +  \frac{1}{\pi^2  t^2 (1+a)} \left( 1 - e^{-\pi   \lambda^{-1}t^2(1-a^2)}\right).
\end{split}
\end{align}
We now choose $1-a^2 = \lambda^{1/2}$ (recall that $\lambda$ in this case is small) and plug it back in \eqref{asymp-4} to get
\begin{equation}\label{asymp-ub}
H_t(\lambda) \leq \frac{\sqrt{1 - \sqrt{\lambda}}}{\pi\lambda} e^{-\pi  \lambda^{-1/2} t^2} + \frac{1}{\pi^2  t^2 \left(1+\sqrt{1 - \sqrt{\lambda}}\right)} \left( 1 - e^{-\pi  \lambda^{-1/2} t^2} \right).
\end{equation}
For fixed $t\neq 0$, as $\lambda \to 0$ we see from expressions \eqref{asymp-lb} and \eqref{asymp-ub} that 
\begin{equation*}
\lim_{\lambda \to 0} H_t(\lambda) = \frac{1}{2 \pi^2 t^2}.
\end{equation*}
Finally, expression \eqref{asymp-dc} allows us to use dominated convergence and conclude that
\begin{align*}
\lim_{\lambda \to 0} H(\lambda) =\lim_{\lambda \to 0}  \sum_{n=-\infty}^{\infty} H_{n+\frac{1}{2}}(\lambda) & = \sum_{n=-\infty}^{\infty} \lim_{\lambda \to 0}  H_{n+\frac{1}{2}}(\lambda)  \\
&  = \frac{1}{2 \pi^2} \sum_{n=-\infty}^{\infty} \frac{1}{\big(n+\tfrac{1}{2}\big)^2} = \frac{1}{2}\,,
\end{align*}
which finishes the proof of \eqref{asymp-lemma-2}.
\end{proof}

\section{Proof of Theorem \ref{thm-app}}\label{App}%%%%%%%%%%%%%%%%%%%%%%%%%%%%%%%%%%%%%%%%%%%%%%%%%%%%%%%%%%%%%%%%%%%%%%%%%%%%%
The strategy for integrating the free parameter in the case of two-sided approximations uses the Paley-Wiener theorem for distributions as in \cite[Sections 7 and 8]{CLV} or \cite[Theorems 1.7.5 and 1.7.7]{Ho}. We start with a more general situation (from which the truncated Gaussian is a particular case) where $\lambda$ is a parameter on an interval $I \subseteq \R$ and $x \mapsto G(\lambda, x)$ is a family of real-valued functions satisfying the following properties, for each $\lambda \in I$,

\medskip

\begin{itemize}
 \item[(i)] The function $x \mapsto G(\lambda, x)$ is continuous on $\R\backslash\{0\}$ and integrable on $\R$.
 
 \medskip

 \item[(ii)] There is a unique best approximation $z\mapsto K(\lambda,z)$ of exponential type $\pi$ that interpolates the values of $x\mapsto G(\lambda,x)$ at $\Z\backslash\{0\}$, and satisfies
\begin{equation*}
 \sin \pi x \,\{G(\lambda,x) - K(\lambda, x)\} \geq 0
\end{equation*}
for all $x \in \R$.
\end{itemize}
We will call $\{x\mapsto G(\lambda, x)\}_{\lambda\in I}$ a \textsl{best approximation family} if it satisfies properties (i) and (ii) above. We denote by $\mc{S}(\R)$ the space of Schwartz functions and by $\mc{S}'(\R)$ the dual space of tempered distributions. In this setting we have the following result.

\begin{lemma}\label{BADis}
Let $\{x \mapsto G(\lambda,x)\}_{\lambda\in I}$ be a best approximation family and $\nu$ be a non-negative Borel measure on $I$ satisfying 
\begin{equation*}
 \int_{I} \int_{-\infty}^{\infty} \left|G(\lambda, x) - K(\lambda, x)\right| \, \dx\,\dnu  < \infty.
\end{equation*}
Let $g: \R \to \R$ be a function on $\mc{S}'(\R)$ that is continuous on $\R\backslash\{0\}$, and such that 
\begin{equation*}
 \widehat{g}(\varphi) = \int_{-\infty}^{\infty}\left\{\int_{I} \widehat{G}(\lambda, t)\, \dnu \right\} \varphi (t) \, \dt
\end{equation*}
(in the tempered distribution sense) for all Schwartz functions $\varphi$ supported on $[-\h,\h]^c$. Then there exists a unique best approximation $k(z)$ of exponential type $\pi$ for $g(x)$. The function $k(x)$ interpolates the values of $g(x)$ at $\Z/\{0\}$ and satisfies 
\begin{equation*}
 \sin \pi x\,\{g(x) - k(x)\} \geq 0,
\end{equation*}
for all $x \in \R$, and
\begin{equation*}
\int_{-\infty}^{\infty} |g(x) -k(x)| \,\dx =  \int_{I} \int_{-\infty}^{\infty} \left|G(\lambda, x) - K(\lambda, x)\right| \, \dx\,\dnu.
\end{equation*}
\end{lemma}

\begin{proof}
The argument is a modification of the proof of \cite[Theorem 16]{CLV}.
\end{proof}

Most of the work towards the proof of Theorem \ref{thm-app} is done. All that remains is to check that the hypotheses of Lemma \ref{BADis} are satisfied in the case of the truncated Gaussian $G_{\lambda}^{+}(x)  = \sgp e^{-\pi \lambda x^2}$. First observe that, by Lemma \ref{asymp-lemma}, for a non-negative Borel measure $\nu$ on $[0,\infty)$ the two conditions
\begin{equation}\label{cond-1}
\int_0^{\infty} \frac{1}{1 + \sqrt{\lambda}}\, \dnu <\infty
\end{equation}
and
\begin{equation*}
 \int_{I} \int_{-\infty}^{\infty} \left|G_{\lambda}^{+}(x) - K_{\lambda}^{+}(x)\right| \, \dx\,\dnu  < \infty
\end{equation*}
are equivalent. It remains to show that the Fourier transform of 
\begin{equation*}
g(x) = \int_0^{\infty}  G_{\lambda}^{+}(x)\, \dnu
\end{equation*}
is given by
\begin{equation*}
\widehat{g}(t) = \int_0^{\infty}  \widehat{G}_{\lambda}^{+}(t)\, \dnu
\end{equation*}
outside a compact $[-\delta, \delta]$ in the tempered distribution sense. For $\delta >0$, let $\varphi:\R \to \R$ be a Schwartz function with support on $[-\delta, \delta]^c$. Using   \eqref{cond-1} we obtain
\begin{align}\label{pre-Fubini}
\begin{split}
\int_{0}^{\infty} \int_{-\infty}^{\infty} &e^{-\pi \lambda x^2} \,\big|\widehat{\varphi}(x)\big| \,\dx \,\dnu  \leq  \int_{0}^{1} \int_{-\infty}^{\infty} \big|\widehat{\varphi}(x)\big| \,\dx\, \dnu \\
& \ \ \ \ \ \ \ \ \ \ \ \ \ \ \ \ \ \ \ \ \ \ \ \ \ \ \ \ \ \ \ + \sup_{x\in \R} \big|\widehat{\varphi}(x)\big| \int_{1}^{\infty} \int_{-\infty}^{\infty} e^{-\pi \lambda x^2}\,\dx \,\dnu \\
&= \int_{-\infty}^{\infty} \big|\widehat{\varphi}(x)\big| \,\dx\, \int_{0}^{1} \dnu + \sup_{x\in \R} \big|\widehat{\varphi}(x)\big| \int_{1}^{\infty} \lambda^{-\frac12}\dnu < \infty.
\end{split}
\end{align}
Also, recall from \eqref{FT-Gaussian-3} that 
\begin{equation}\label{II-FT-Gaussian-3}
\widehat{G}_{\lambda}^{+}(t) = \frac{1}{2} \, \lambda^{-1/2} e^{-\pi \lambda^{-1} t^2} + \frac{t}{i \lambda}  \int_0^1e^{-\pi \lambda^{-1} t^2 (1-y^2)}\,\dy.
\end{equation}
From \eqref{asymp-dc} and  \eqref{II-FT-Gaussian-3} we know that 
\begin{equation} \label{pr-th4-2}
 \big|\widehat{G}_{\lambda}^{+}(t)\big| \leq  \frac{C_3}{ |t|}\,,
\end{equation}
for some $C_3>0$, and directly from \eqref{II-FT-Gaussian-3} we also see that
\begin{equation}\label{pr-th4-3}
 \big|\widehat{G}_{\lambda}^{+}(t)\big|  \leq \frac{1}{2\sqrt{\lambda}}+  \frac{|t|}{\lambda}.
\end{equation}
Expressions \eqref{pr-th4-2} and \eqref{pr-th4-3} combine to give 
\begin{align*}
\int_0^{\infty} \big|\widehat{G}_{\lambda}^{+}(t)\big| \, \dnu  &= \int_0^{1} \big|\widehat{G}_{\lambda}^{+}(t)\big|\,  \dnu + \int_1^{\infty} \big|\widehat{G}_{\lambda}^{+}(t)\big|\,  \dnu \\
& \leq \frac{C_3}{|t|} \int_0^{1}  \dnu +  \int_1^{\infty} \left\{\frac{1}{2\sqrt{\lambda}}+  \frac{|t|}{\lambda}\right\} \, \dnu \leq \frac{C_4}{|t|} + C_5 + C_6|t|\,,
\end{align*}
where the constants $C_4$, $C_5$ and $C_6$ depend only on $\nu$. This verifies that (recall that $\varphi$ vanishes near the origin)
\begin{equation}\label{pre-Fubini-2}
 \int_{-\infty}^{\infty} \int_{0}^{\infty}  \big|\widehat{G}_{\lambda}^{+}(t)\big|\, |\varphi(t)|\,\dnu\, \dt \leq  \int_{-\infty}^{\infty} \left(\frac{C_4}{|t|} + C_5 + C_6|t|\right) \, |\varphi(t)|\, \dt < \infty.
\end{equation}
Plainly, expressions  \eqref{pre-Fubini} and \eqref{pre-Fubini-2} allow us to apply Fubini's theorem twice in the computation below
\begin{align*}
\int_{-\infty}^{\infty} g(x)\, \widehat{\varphi}(x)\, \dx & =    \int_{-\infty}^{\infty} \int_0^{\infty}  G_{\lambda}^{+}(x)\, \widehat{\varphi}(x)\,\dnu\, \dx\\
& =    \int_{0}^{\infty} \int_{-\infty}^{\infty}  G_{\lambda}^{+}(x)\, \widehat{\varphi}(x)\,\dx\, \dnu \\
&  =  \int_{0}^{\infty} \int_{-\infty}^{\infty}  \widehat{G}_{\lambda}^{+}(t)\, \varphi(t)\,\dt\, \dnu \\
&= \int_{-\infty}^{\infty} \left\{ \int_{0}^{\infty}  \widehat{G}_{\lambda}^{+}(t)\, \dnu \right\}\varphi(t)\,\dt\,,
\end{align*}
which gives the required characterization of the Fourier transform $\widehat{g}(t)$ outside the origin in the distribution sense and completes the proof of the Theorem \ref{thm-app}.

\section{Proof of Theorem \ref{thm-app-one-sided}}%%%%%%%%%%%%%%%%%%%%%%%%%%%%%%%%%%%%%%%%%%%%%%%%%%%%%%%%%%%%%%%%%%%%%%%%%%%%%%%%%%%%%%%%%%%%%%%%%%%%%%%%%%%%%%%%%%%%%%%%%%%%%%%%

In the one-sided case a more straightforward approach of moving the integral inside the summation series and guaranteeing its absolute convergence will do the job. We start with part (i), the minorant case, where we proved that 
\begin{equation*}
\lplus(z) = \frac{\sin^2\pi z}{\pi^2} \sum_{n=1}^\infty \left\{\frac{G_\lambda(n)}{(z-n)^2} + \frac{G_\lambda'(n)}{z-n}\right\} -\frac{\sin^2\pi z}{\pi^2 z} \sum_{n=1}^\infty G_\lambda'(n)
\end{equation*}
satisfies
\begin{equation}\label{pf4-1}
\lplus(x) \leq \gplus(x)
\end{equation}
for all $x \in \R$, with 
\begin{equation}\label{pf4-2}
\lplus(n) = \gplus(n)
\end{equation}
if $n \in \Z/\{0\}$, and 
\begin{equation}\label{pf4-3}
\lplus(0)  = \lim_{x \to 0^{-}}\gplus(x) = 0.
\end{equation}
We consider a non-negative Borel measure $\nu$ satisfying \eqref{nu-1} and we need to show that 
$$l(z) = \int_0^{\infty} \lplus(z)\, \dnu$$
is a well defined entire function of exponential type at most $2\pi$. If this is the case, by integrating expressions \eqref{pf4-1}, \eqref{pf4-2} and \eqref{pf4-3} against $\nu$, these properties will be carried on to $l(x)$ and $g(x) =  \int_0^{\infty} \gplus(x)\, \dnu$ making $l(x)$ the unique extremal minorant of exponential type at most $2\pi$ for $g(x)$ via the same arguments used in the proof of Theorem \ref{thm-one-sided}.

\medskip

For this purpose we need to collect some estimates. For $n \in \N$ using \eqref{nu-1} we have
\begin{equation}\label{pf4-est1}
\int_0^{\infty} G_\lambda(n)\,\dnu = \int_0^{1} G_\lambda(n)\,\dnu  + \int_1^{\infty} \sqrt{\lambda}\,G_\lambda(n)\,\frac{\dnu}{\sqrt{\lambda}} \leq C_1 + \frac{C_2}{n},
\end{equation}
and
\begin{align}\label{pf4-est2}
\begin{split}
\int_0^{\infty} \big|G'_\lambda(n)\big|\,\dnu &= 2 \pi \int_0^{1} \lambda \,n\, G_\lambda(n)\,\dnu  + 2 \pi \int_1^{\infty} \lambda^{3/2}\,n \,G_\lambda(n)\,\frac{\dnu}{\sqrt{\lambda}} \\
& \leq \frac{C_3}{n} + \frac{C_4}{n^2},
\end{split}
\end{align}
where $C_1, C_2, C_3$ and $C_4$ are positive constants depending exclusively on $\nu$.

\medskip

To analyze the remaining term observe that
\begin{align*}
\lambda ^{1/2} \sum_{n=1}^{\infty} \big|G'_\lambda(n)\big| & = \sum_{n=1}^{\infty}\frac{2 \pi}{n^2} \,\lambda^{3/2}\,n^3\, G_\lambda(n) \leq C_5 \sum_{n=1}^{\infty}\frac{2 \pi}{n^2},
\end{align*}
which proves that $\sum_{n=1}^{\infty} \big|G'_\lambda(n)\big| $ is $\mathcal{O}\big(\lambda^{-1/2}\big)$ as $ \lambda \to \infty$. On the other hand, using the arithmetic-geometric mean inequality and \eqref{ineq-20}, we also obtain
\begin{align*}
\begin{split}
\sum_{n=1}^{\infty} \big|G'_\lambda(n)\big| & = \sum_{n=1}^{\infty} 2 \pi \lambda\,n\, G_\lambda(n) \leq \sum_{n=1}^{\infty}  \pi\, \big\{ \lambda^{3/2}\,n^2 + \lambda^{1/2}\big\}\, G_\lambda(n)\\
& \leq \frac{\lambda^{1/2}}{4} +\left(\tfrac12 + \pi\right)\, \lambda^{1/2} \, \sum_{n=1}^{\infty} G_\lambda(n)\\
& =  \frac{\lambda^{1/2}}{4} +\left(\tfrac12 + \pi\right)\, \lambda^{1/2} \left(\frac{\theta_3(0,i\lambda) -1}{2} \right).
\end{split}
\end{align*}
We know $\theta_3(0,i\lambda) \to \lambda^{-1/2}$ as $\lambda \to 0$, by the transformation formula \eqref{th3-series}. Therefore we may conclude that $\sum_{n=1}^{\infty} \big|G'_\lambda(n)\big| $ is $\mathcal{O}(1)$ as $ \lambda \to 0$. 

\medskip

This shows that  $\sum_{n=1}^{\infty} \big|G'_\lambda(n)\big| $ is $\nu$-integrable, and together with \eqref{pf4-est1} and \eqref{pf4-est2} we can can move the integration inside the summation series since it converges absolutely to obtain
\begin{align*}
\begin{split}
l(z) &= \int_0^{\infty} \lplus(z)\, \dnu \\
& =  \frac{\sin^2\pi z}{\pi^2} \sum_{n=1}^\infty \left\{\frac{\int_0^{\infty}G_\lambda(n)\,\dnu}{(z-n)^2} + \frac{\int_0^{\infty}G_\lambda'(n)\,\dnu}{z-n}\right\}\\
&  \ \ \ \ \ \ \ \ \ \ \ \ \ \ \ \ \ \ \ \ \ \ \ \ \  -\frac{\sin^2\pi z}{\pi^2 z} \int_0^{\infty} \sum_{n=1}^\infty G_\lambda'(n)\,\dnu.
\end{split}
\end{align*}
An application of Morera's theorem shows that this is an entire function and the exponential type $2\pi$ is given by the main term $\sin^2\pi z$. The proof of the majorizing case is analogous.

\section*{Acknowledgments}
The authors are thankful to Marian Bocea for helpful discussions regarding the maximum principle of the heat operator and to Jeffrey D. Vaaler for the discussions on the extremal problem. E. Carneiro acknowledges support from the Institute for Advanced Study via the National Science Foundation agreement No. DMS-0635607 and support from the CNPq-Brazil grants $473152/2011-8$  and $302809/2011-2$.

\end{document}